\documentclass[10pt]{article}

\usepackage{amsmath,amssymb,amscd,amsthm,makeidx,txfonts,graphicx,url,bm}
\usepackage{a4wide,pstricks,xy}

\numberwithin{equation}{subsection}

\let\oldmarginpar\marginpar
\renewcommand\marginpar[1]{\-\oldmarginpar[\raggedleft\footnotesize #1]
{\raggedright\footnotesize #1}}

\makeindex

\xyoption{all}
\input{xypic}

\newcommand{\conjeq}{\stackrel{?}{=}}

\newtheorem{theorem}{Theorem}[subsection]
\newtheorem{proposition}[theorem]{Proposition}
\newtheorem{corollary}[theorem]{Corollary}
\newtheorem{conjecture}[theorem]{Conjecture}
\newtheorem{lemma}[theorem]{Lemma}

\theoremstyle{remark}
\newtheorem{remark}[theorem]{Remark}

\theoremstyle{definition}
\newtheorem{definition}[theorem]{Definition}

\newcounter{margin}
{\end{itshape}  \bigskip}

\def\beq{\begin{eqnarray}}
\def\eeq{\end{eqnarray}}
\def\bes{\begin{eqnarray*}}
\def\ees{\end{eqnarray*}}

\def\muhat{{\bm \mu}}

\def\lambdahat{{\bm \lambda}}

\def\alphahat{{\bm \alpha}}
\def\nuhat{{\bm \nu}}
\def\omhat{{\bm \omega}}
\def\varphihat{{\bm \varphi}}
\def\Vhat{{\bm V}}

\def\X{{\mathbb{X}}}

\def\V{\mathcal{V}}
\def\C{\mathbb{C}}
\def\M{{\mathcal{M}}}
\def\calQ{{\mathcal{Q}}}
\def\calA{{\mathcal{A}}}

\def\calX{{\mathcal{X}}}
\def\calU{{\mathcal{U}}}
\def\calV{{\mathcal{V}}}
\def\calF{{\mathcal{F}}}

\def\calP{\mathcal{P}}
\def\calH{\mathcal{H}}

\def\x{\mathbf{x}}
\def\e{\mathbf{e}}
\def\y{\mathbf{y}}
\def\k{\mathbf{k}}
\def\v{\mathbf{v}}
\def\w{\mathbf{w}}
\def\P{\mathcal{P}}
\def\H{\mathbb{H}}
\def\N{\mathbb{Z}_{\geq 0}}
\def\Nstar{\mathbb{Z}_{> 0}}

\def\F{\mathbb{F}}
\def\Q{\mathbb{Q}}

\def\calC{{\mathcal C}}
\def\calO{{\mathcal O}}

\def\Z{\mathbb{Z}}
\def\nrm{\sigma}

\def\K{\mathbb{K}}

\def\gl{{\mathfrak g\mathfrak l}}

\newcommand{\nc}{\newcommand}

\def\calZ{\mathcal{Z}}

\nc{\op}[1]{\mathop{\mathchoice{\mbox{\rm #1}}{\mbox{\rm #1}}
{\mbox{\rm \scriptsize #1}}{\mbox{\rm \tiny #1}}}\nolimits}
\nc{\al}{\alpha}

\nc{\ep}{\varepsilon} \nc{\ga}{\gamma} \nc{\Ga}{\Gamma}
\nc{\la}{\lambda} \nc{\La}{\Lambda} \nc{\si}{\sigma}
\nc{\Sig}{{\Gamma}} \nc{\Om}{\Omega} \nc{\om}{\omega}

\nc{\SL}{{\rm SL}} \nc{\GL}{{\rm GL}} \nc{\PGL}{{\rm PGL}}
\nc{\G}{{\rm G}}

\nc{\cpt}{{\op{cpt}}} \nc{\Dol}{{\op{Dol}}} \nc{\DR}{{\op{DR}}}
\nc{\B}{{\op{B}}} \nc{\Triv}{\op{Triv}} \nc{\Hod}{{\op{Hod}}}
\nc{\Log}{{\op{Log}}} \nc{\Exp}{{\op{Exp}}} \nc{\Est}{E_{\op{st}}}
\nc{\Hst}{H_{\op{st}}} \nc{\Left}[1]{\hbox{$\left#1\vbox to
  10.5pt{}\right.\nulldelimiterspace=0pt \mathsurround=0pt$}}
\nc{\Right}[1]{\hbox{$\left.\vbox to
  10.5pt{}\right#1\nulldelimiterspace=0pt \mathsurround=0pt$}}
\nc{\LEFT}[1]{\hbox{$\left#1\vbox to
  15.5pt{}\right.\nulldelimiterspace=0pt \mathsurround=0pt$}}
\nc{\RIGHT}[1]{\hbox{$\left.\vbox to
  15.5pt{}\right#1\nulldelimiterspace=0pt \mathsurround=0pt$}}

\nc{\bee}{{\bf E}} \nc{\bphi}{{\bf \Phi}}

\begin{document}

\title{Arithmetic harmonic analysis on\\ character and quiver varieties II}

\author{ Tam\'as Hausel
\\ {\it University of Oxford}
\\{\tt hausel@maths.ox.ac.uk} \and Emmanuel Letellier \\ {\it
  Universit\'e de Caen} \\ {\tt
  letellier.emmanuel@math.unicaen.fr}\and Fernando Rodriguez-Villegas
\\ 
{\it University of Texas at Austin} \\ {\tt
  villegas@math.utexas.edu}\\ \\
{with an appendix by Gergely~Harcos} 
 }

\pagestyle{myheadings}

\maketitle

\begin{abstract}We study connections between the topology of generic
  character varieties of fundamental groups of punctured Riemann surfaces, Macdonald polynomials, quiver representations,
  Hilbert schemes on $\C^\times\times\C^\times$, modular forms and multiplicities in tensor products of irreducible
  characters of finite general linear groups.
\end{abstract}

\newpage
\tableofcontents
\newpage

\section{Introduction}

\subsection{Character varieties}\label{char}
Given a non-negative integer $g$ and a $k$-tuple
$\muhat=(\mu^1,\mu^2,\dots,\mu^k)$ of partitions of $n$, we define the
generic character variety $\M_\muhat$ of type $\muhat$ as
follows (see \cite{hausel-letellier-villegas} for more details). Choose a \emph{generic} tuple $(\calC_1,\dots,\calC_k)$ of
semisimple conjugacy classes of $\GL_n(\C)$ such that for each
$i=1,2,\dots,k$ the multiplicities of the eigenvalues of $\calC_i$ are
given by the parts of $\mu^i$.

Define $\cal{Z}_\muhat$ as 
$$
\calZ_\muhat:= \left\{(a_1,b_1,\dots,a_g,b_g,x_1,\dots,x_k)\in
  (\GL_n)^{2g} \times\calC_1\times\cdots\times\calC_k\,\left|\,
    \prod_{j=1}^g(a_i,b_i)\prod_{i=1}^kx_i=1\right.\right\},
$$
where $(a,b)=aba^{-1}b^{-1}$. The group $\GL_n$ acts diagonally by
conjugation on $\calZ_\muhat$ and we define $\M_\muhat$ as the affine
GIT quotient 
$$
\M_\muhat:=\calZ_\muhat/\!/\GL_n:={\rm
  Spec}\,\left(\C[\calZ_\muhat]^{\GL_n}\right).
$$ 
We prove in \cite{hausel-letellier-villegas} that, if non-empty,
$\M_\muhat$ is nonsingular of pure dimension
$$
d_\muhat:=n^2(2g-2+k)-\sum_{i,j}(\mu^i_j)^2+2.
$$
We also defined an {\em a priori} rational function
$\H_\muhat(z,w)\in\Q(z,w)$ in terms of  Macdonald symmetric
functions (see \S~\ref{Cauchy} for a precise definition) and we conjecture
that the compactly supported mixed Hodge numbers $\{h_c^{i,j;k}(\M_\muhat)\}_{i,j,k}$ satisfies $h_c^{i,j;k}(\M_\muhat)=0$ unless $i=j$ and 
\begin{equation}
  H_c(\M_\muhat;q,t)\conjeq(t\sqrt{q})^{d_\muhat}\H_\muhat 
\left(-t\sqrt{q},\frac{1}{\sqrt{q}}\right),
\label{mainconj}
\end{equation}
where $H_c(\M_\muhat;q,t):=\sum_{i,j}h_c^{i,i;j}(\M_\muhat)q^it^j$ is
the compactly supported mixed Hodge polynomial. 

In particular, $\H_\muhat(-z,w)$
should actually be a polynomial with non-negative integer coefficients
of degree $d_\muhat$ in each variable.  

In
\cite{hausel-letellier-villegas} we prove that (\ref{mainconj}) is
true under the specialization $(q,t)\mapsto (q,-1)$, namely,
\begin{equation}
  E(\M_\muhat;q):=H_c(\M_\muhat;q,-1)= 
q^{\frac{1}{2}d_\muhat}\H_\muhat\left(\sqrt{q},\frac{1}{\sqrt{q}}\right).
\label{mainresult}
\end{equation}

This formula is obtained by counting points of $\M_\muhat$ over finite
fields (after choosing a spreading out of $\M_\muhat$ over a finitely
generated subalgebra of $\C$). We compute $\M_\muhat(\F_q)$ using a
formula involving the values of the irreducible characters of
$\GL_n(\F_q)$ (a formula that goes back to Frobenius
\cite{frobenius}).  The calculation shows that $\M_\muhat$ is
\emph{polynomial count}; i.e., there exists a polynomial $P\in\C[T]$
such that for any finite field $\F_q$ of sufficiently large
characteristic, $\#\M_\muhat(\F_q)=P(q)$. Then by a theorem of
Katz \cite[Appendix]{hausel-letellier-villegas} $E(\M_\muhat;q)=P(q)$.

Recall also that the $E(\M_\muhat;q)$ satisfies the following identity

\begin{equation}
E(\M_\muhat;q)=q^{d_\muhat}E(\M_\muhat;q^{-1}).
\label{curious}\end{equation}

In this paper we use Formula (\ref{mainresult}) to prove the following theorem.

\begin{theorem} 
If non-empty, the character variety $\M_\muhat$ is connected.
\end{theorem}
The proof of the theorem reduces to proving that the coefficient of
the lowest power of $q$ in $\H_\muhat(\sqrt{q},1/\sqrt{q})$, namely
$q^{-d_\muhat/2}$, equals $1$.  This turns out to require a rather
delicate argument, by far the most technical of the paper, that uses
the inequality of \S~\ref{appendix} in a crucial way.

\subsection{Relations to Hilbert schemes on $\C^\times\times\C^\times$ and modular forms}

Here we assume that $g=k=1$. Put $X=\C^\times\times\C^\times$ and
denote by $X^{[n]}$ the Hilbert scheme of $n$ points in $X$.  Define
 $\H^{[n]}(z,w)\in \Q(z,w)$ by
\begin{equation}
\sum_{n\geq 0}\H^{[n]}(z,w)T^n:=\prod_{n\geq
1}\frac{(1-zwT^n)^2}{(1-z^2T^n)(1-w^2T^n)},
\end{equation}
with the convention that $\H^{[0]}(z,w):=1$.  It is known
by work of G\"ottsche and Soergel \cite{Gottsche-Soergel} that the
mixed Hodge polynomial $H_c\left(X^{[n]};q,t\right)$
is given by
$$
H_c\left(X^{[n]};q,t\right)
=(qt^2)^n\H^{[n]}\left(-t\sqrt{q}, \frac{1}{\sqrt{q}}\right).
$$

\begin{conjecture} 
We have $$\H^{[n]}(z,w)=\H_{(n-1,1)}(z,w).$$
\label{conjHS}
\end{conjecture}

This together with the conjectural formula (\ref{mainconj}) implies
that the Hilbert scheme $X^{[n]}$ and the character variety
$\M_{(n-1,1)}$ should have the same mixed Hodge polynomial. Although
this is believed to be true (in the analogous additive case this is
well-known; see Theorem \ref{adpure}) there is no complete proof in
the literature.  (The result follows from known facts modulo some
missing arguments in the non-Abelian Hodge theory for punctured
Riemann surfaces; see the comment after Conjecture \ref{conjCV=HS}.)
We prove the following results which give evidence for
Conjecture~\ref{conjHS}.

\begin{theorem} We have
\begin{align*}
&\H^{[n]}(0,w)=\H_{(n-1,1)}\left(0,w\right),\\
&\H^{[n]}(w^{-1},w)=\H_{(n-1,1)}(w^{-1},w).
\end{align*}
\label{theospe}
\end{theorem}
The second identity means that the $E$-polynomials of $X^{[n]}$ and
$\M_{(n-1,1)}$ agree.  As a consequence of Theorem \ref{theospe} we
have the following relation between character varieties and
quasi-modular forms.
\begin{corollary}
We have 
$$
1+\sum_{n\geq 1}\H_{(n-1,1)}\left(e^{u/2},e^{-u/2}\right)T^n
=\frac{1}{u}\left(e^{u/2}-e^{-u/2}\right)\exp\left(2\sum_{k\geq
    2}G_k(T)\frac{u^k}{k!}\right),
$$
where $$G_k(T)=\frac{-B_k}{2k}+\sum_{n\geq 1}\sum_{d\,|\, n}
d^{k-1}T^n$$(with $B_k$ is the $k$-th Bernoulli number) is the
classical Eisenstein series for $SL_2(\Z)$.

In particular, the coefficient of any power of $u$ in the left hand
side is in the ring of \emph{quasi-modular} forms, generated by the
$G_k$, $k\geq 2$, over $\Q$.
\end{corollary}

Relation between Hilbert schemes and modular forms was first investigated by G\"ottsche \cite{Gottsche}. 

\subsection{Quiver representations}
\label{quiver-repns}
For a partition $\mu=\mu_1\geq\dots \geq \mu_r>0$ of $n$ we denote by
$l(\mu)=r$ its length. Given a non-negative integer $g$ and a
$k$-tuple $\muhat=(\mu^1,\mu^2,\dots,\mu^k)$ of partitions of $n$ we
define a \emph{comet-shaped} quiver $\Gamma_\muhat$ with $k$ legs of
length $s_1,s_2,\dots,s_k$ (where $s_i=l(\mu^i)-1$) and with $g$ loops
at the central vertex (see picture in \S \ref{comet}). The
multi-partition $\muhat$ defines also a dimension vector $\v_\muhat$
of $\Gamma_\muhat$ whose coordinates on the $i$-th leg are
$(n,n-\mu^i_1,n-\mu^i_1-\mu^i_2,\dots,n-\sum_{r=1}^{s_i}\mu^i_r)$.

By a theorem of Kac \cite{kacconj} there exists a monic polynomial
$A_\muhat(T)\in \Z[T]$ of degree $d_\muhat/2$ such that the number of
absolutely indecomposable representations over $\F_q$ (up to
isomorphism) of $\Gamma_\muhat$ of dimension $\v_\muhat$ equals
$A_\muhat(q)$.

Let us state our first main result.
\begin{theorem}
We have \begin{equation}A_\muhat(q)=\H_\muhat(0,\sqrt{q}).
\label{purity0}
\end{equation}
\end{theorem}

If we assume that $\v_\muhat$ is indivisible, i.e., the gcd of all the
parts of the partitions $\mu^1,\dots,\mu^k$ equals $1$, then, as
mentioned in \cite[Remark 1.4.3]{hausel-letellier-villegas}, the
formula can be proved using the results of Crawley-Boevey and van den
Bergh \cite{crawley-boevey-etal} together with the results in
\cite{hausel-letellier-villegas}. More precisely the results of
Crawley-Boevey and van den Bergh say that $A_\muhat(q)$ equals (up to
some power of $q$) the compactly supported Poincar\'e polynomial of
some quiver variety $\calQ_\muhat$ (which exists only if $\v_\muhat$
is indivisible). In \cite{hausel-letellier-villegas} we show that the
Poincar\'e polynomial of $\calQ_\muhat$ agrees with
$\H_\muhat(0,\sqrt{q})$ up to the same power of $q$, hence the formula
(\ref{purity0}).

The proof of Formula (\ref{purity0}) we give in this paper is
completely combinatorial (and works also in the divisible case). It is
based on Hua's formula \cite{hua} for the number of absolutely
indecomposable representations of quivers over finite fields.

The conjectural formula (\ref{mainconj}) together with Formula (\ref{purity0}) implies the following conjecture.

\begin{conjecture}
  We  have $$A_\muhat(q)=q^{-\frac{d_\muhat}{2}}PH_c(\M_\muhat;q),$$where
  $PH_c(\M_\muhat;q):=\sum_ih_c^{i,i;2i}(\M_\muhat)q^i$ is the
  \emph{pure part} of
  $H_c(\M_\muhat;q,t)$.
\label{purconj}
\end{conjecture} 
Conjecture \ref{purconj} implies Kac's conjecture \cite{kacconj} for
comet shaped quivers, namely, $A_\muhat(q)$ is a polynomial in $q$
with non-negative coefficients (see \S \ref{genquiv} for more
details).

\subsection{Characters of general linear groups over finite fields}

Given two irreducible complex characters $\calX_1,\calX_2$ of
$\GL_n(\F_q)$ it is a natural and difficult question to understand the
decomposition of the tensor product $\calX_1\otimes\calX_2$ as a sum
of irreducible characters.  Note that the character table of
$\GL_n(\F_q)$ is known (Green, 1955) and so we can compute in theory
the multiplicity $\langle\calX_1\otimes\calX_2,\calX\rangle$ of any
irreducible character $\calX$ of $\GL_n(\F_q)$ in
$\calX_1\otimes\calX_2$ using the scalar product
formula \begin{equation}\langle\calX_1\otimes\calX_2,\calX\rangle=\frac{1}{|\GL_n(\F_q)|}\sum_{g\in\GL_n(\F_q)}\calX_1(g)\calX_2(g)\overline{\calX(g)}.\label{scalprod}\end{equation}However
it is very difficult to extract any interesting information from this
formula. In his thesis Mattig uses this formula to compute (with the
help of a computer) the multiplicities $\langle
\calX_1\otimes\calX_2,\calX\rangle$ when $\calX_1,\calX_2,\calX$ are
\emph{unipotent characters} and when $n\leq 8$ (see \cite{Hiss}), and
he noticed that $\langle \calX_1\otimes\calX_2,\calX\rangle$ is a
polynomial in $q$ with positive integer coefficients.

In \cite{hausel-letellier-villegas} we define the notion of
\emph{generic} tuple $(\calX_1,\dots,\calX_k)$ of irreducible
characters of $\GL_n(\F_q)$. We also consider the character
$\Lambda:\GL_n(\F_q)\rightarrow \C$, $x\mapsto q^{g\cdot {\rm dim}\,
  C_{\GL_n}(x)}$ where $C_{\GL_n}(x)$ denotes the centralizer of $x$
in $\GL_n(\overline{\F}_q)$ and where $g$ is a non-negative
integer. If $g=1$, this is the character of the conjugation action of
$\GL_n(\F_q)$ on the group algebra $\C[\gl_n(\F_q)]$.

If $\mu=(\mu_1,\mu_2,\dots,\mu_r)$ is a partition of $n$, an
irreducible character of $\GL_n(\F_q)$ is said to be of type $\mu$ if
it is of the form $R_{L_\mu}^{GL_n}(\alpha)$ where
$L_\mu=\GL_{\mu_1}\times\GL_{\mu_2}\times\cdots\times\GL_{\mu_r}$ and
where $\alpha$ is a \emph{regular} linear character of $L_\mu(\F_q)$,
see \S \ref{applichar} for definitions. Characters of this form are
called \emph{semisimple split}.

In \cite{hausel-letellier-villegas} we prove that for a generic tuple
$(\calX_1,\dots,\calX_k)$ of semisimple split irreducible characters
of $\GL_n(\F_q)$ of type $\muhat$, we have 
\begin{equation}
\langle
  \Lambda\otimes\calX_1\otimes\cdots\otimes\calX_k,1\rangle
  =\H_\muhat(0,\sqrt{q})\label{multi0}.
\end{equation} 
Note that in particular this implies that the left hand side only
depends on the combinatorial type $\muhat$  not on the specific choice
of characters.

Together with Formula (\ref{purity0}) we deduce the following formula.

\begin{theorem}We have $$\langle
  \Lambda\otimes\calX_1\otimes\cdots\otimes\calX_k,1\rangle=A_\muhat(q).$$
\end{theorem}

Using Kac's results on quiver representations (see \S \ref{genquiv})
the above theorem has the following consequence. 

\begin{corollary}Let $\Phi(\Gamma_\muhat)$ denote the root system
  associated with $\Gamma_\muhat$ and let $(\calX_1,\dots,\calX_k)$ be
  a generic $k$-tuple of irreducible characters of $\GL_n(\F_q)$ of
  type $\muhat$.

  We have $\langle
  \Lambda\otimes\calX_1\otimes\cdots\otimes\calX_k,1\rangle\neq 0$ if
  and only if $\v_\muhat\in\Phi(\Gamma_\muhat)$. Moreover $\langle
  \Lambda\otimes\calX_1\otimes\cdots\otimes\calX_k,1\rangle=1$ if and
  only if $\v_\muhat$ is a real root.
\label{coromulti0}
\end{corollary}

In \cite{letellier} the second author proves that Corollary \ref{coromulti0} extends to any type of generic
tuples of irreducible characters of $\GL_n(\F_q)$ (not necessarily
semisimple split).

Recall that there is a natural parametrization, $\mu\mapsto
\calU_\mu$, of the unipotent characters of $\GL_n(\F_q)$ by partitions
of $n$, fixed by requiring that $\calU_{(1^n)}$ is trivial.  Using
again quiver representations, we also prove in~\S\ref{applichar} the following
result on multiplicities for unipotent characters (confirming partly
Mattig's observation for $n\leq 8$).

\begin{proposition}
 For a $k$-tuple of partitions
  $\muhat=(\mu^1,\ldots,\mu^k)$ of $n$ there exists a polynomial
  $U_\muhat\in \Z[T]$ such that $\langle
  \Lambda\otimes \calU_{\mu^1}\otimes\cdots\otimes
  \calU_{\mu^k},1\rangle=U_\muhat(q)$. 
\end{proposition}

\paragraph{Acknowledgements.} We would like to thank the Mathematisches
Forschungsinstitut Oberwolfach for a research in pairs stay where much
of the work was done. TH was supported by a Royal Society University Research Fellowship. EL was supported by ANR-09-JCJC-0102-01.  FRV
was supported by NSF grant DMS-0200605, an FRA from the University of
Texas at Austin, EPSRC grant EP/G027110/1, Visiting Fellowships at All
Souls and Wadham Colleges in Oxford and a Research Scholarship from
the Clay Mathematical Institute.

\section{Preliminaries}

We denote by $\F$ an algebraic closure of a finite field $\F_q$.

\subsection{Symmetric functions}

\subsubsection{Partitions, Macdonald polynomials, Green polynomials}

We denote by $\calP$ the set of all partitions including the unique
partition $0$ of $0$, by $\calP^*$ the set of non-zero partitions and
by $\calP_n$ be the set of partitions of $n$. Partitions $\lambda$ are
denoted by $\lambda=(\lambda_1,\lambda_2,\ldots)$, where
$\lambda_1\geq \lambda_2\geq\cdots\geq 0$. We will also sometimes
write a partition as $(1^{m_1},2^{m_2},\ldots,n^{m_n})$ where $m_i$
denotes the multiplicity of $i$ in $\lambda$. The {\it size} of
$\lambda$ is $|\lambda|:=\sum_i\lambda_i$; the {\it length}
$l(\lambda)$ of $\lambda$ is the maximum $i$ with $\lambda_i>0$. For
two partitions $\lambda$ and $\mu$, we define
$\langle\lambda,\mu\rangle$ as $\sum_i\lambda_i'\mu_i'$ where
$\lambda'$ denotes the dual partition of $\lambda$. We put
$n(\lambda)=\sum_{i>0}(i-1)\lambda_i$. Then
$\langle\lambda,\lambda\rangle=2n(\lambda)+|\lambda|$.  For two
partitions $\lambda=(1^{n_1},2^{n_2},\dots)$ and
$\mu=(1^{m_1},2^{m_2},\dots)$, we denote by $\lambda\cup\mu$ the
partition $(1^{n_1+m_1},2^{n_2+m_2},\dots)$. For a non-negative integer
$d$ and a partition $\lambda$, we denote by $d\cdot\lambda$ the
partition $(d\lambda_1,d\lambda_2,\dots)$. The {\it dominance ordering}
for partitions is defined as follows: $\mu \unlhd \lambda$ if and only
if $\mu_1+\cdots+\mu_j\leq \lambda_1+\cdots+\lambda_j$ for all
$j\geq 1$. 

Let $\x=\{x_{1},x_{2},\dots\}$ be an infinite set of variables and
$\Lambda(\x)$ the corresponding ring of symmetric functions.  As usual we will denote by $s_\lambda(\x),h_\lambda(\x),p_\lambda(\x)$, and $m_\lambda(\x)$, the Schur symmetric functions, the complete symmetric functions, the power symmetric functions and the monomial symmetric functions.

We will
deal with elements of the ring $\Lambda(\x) \otimes_\Z \Q(z,w)$ and
their images under two specializations: their {\it pure part},
$z=0,w=\sqrt q$ and their {\it Euler specialization}, $z=\sqrt
q,w=1/\sqrt q$.

For a partition $\lambda$, let $\tilde{H}_\lambda(\x;q,t) \in
\Lambda(\x) \otimes_\Z \Q(q,t)$ be the {\it Macdonald symmetric
  function} defined in Garsia and Haiman \cite[I.11]{garsia-haiman}.  We collect in this
section some basic properties of these functions that we will need.

We have the duality
\begin{equation}
\tilde{H}_\lambda(\x;q,t)=\tilde{H}_{\lambda^\prime}(\x;t,q)
\label{Hduality}\end{equation}
see \cite[Corollary 3.2]{garsia-haiman}. We define the (transformed)
{\it Hall-Littlewood symmetric function} as
\begin{equation}
\label{HL-defn}
\tilde{H}_\lambda(\x;q):=\tilde{H}_\lambda(\x;0,q).
\end{equation}
In the notation just introduced then $\tilde{H}_\lambda(\x;q)$ is the
pure part of $\tilde{H}_\lambda(\x;z^2,w^2)$.

Under the Euler specialization of $\tilde{H}_\lambda(\x;z^2,w^2)$ we have \cite[Lemma 2.3.4]{hausel-letellier-villegas} \begin{equation}\tilde{H}_\lambda(\x;q,q^{-1})=q^{-n(\lambda)}H_\lambda(q)s_\lambda(\x\y),\label{eulermac}\end{equation}where $y_i=q^{i-1}$ and $H_\lambda(q):=\prod_{s\in\lambda}(1-q^{h(s)})$ is the \emph{hook polynomial} \cite[I, 3, example 2]{macdonald}.

Define the {\it $(q,t)$-Kostka polynomials}
$\tilde{K}_{\nu\lambda}(q,t)$ by
\begin{equation}
\label{K-defn}
\tilde{H}_{\lambda}(\x;q,t)=
\sum_{\nu}\tilde{K}_{\nu\lambda}(q,t)s_{\nu}(\x).
\end{equation}
These are $(q,t)$
generalizations of the $\tilde{K}_{\nu\lambda}(q)$ Kostka-Foulkes 
polynomial in Macdonald \cite[III, (7.11)]{macdonald}, which are obtained as
$q^{n(\lambda)}K_{\nu\lambda}(q^{-1})=
\tilde{K}_{\nu\lambda}(q)=\tilde{K}_{\nu\lambda}(0,q)$, i.e., by 
taking their pure part. In particular,
\begin{equation}
\label{Hall-Littlewood}
\tilde{H}_{\lambda}(\x;q)=\sum_{\nu}\tilde{K}_{\nu\lambda}(q)s_{\nu}(\x).
\end{equation}

For a partition $\lambda$, we denote by $\chi^\lambda$ the corresponding irreducible character of $S_{|\lambda|}$ as in Macdonald \cite{macdonald}. Under this parameterization, the character $\chi^{(1^n)}$ is the sign character of $S_{|\lambda|}$ and $\chi^{(n^1)}$ is the trivial character. Recall also that the decomposition into disjoint cycles provides a natural parameterization of the conjugacy classes of $S_n$ by the partitions of $n$. We then denote by $\chi^\lambda_\mu$ the value of $\chi^\lambda$ at the conjugacy class of $S_{|\lambda|}$ corresponding to $\mu$ (we use the convention that $\chi^\lambda_\mu=0$ if $|\lambda|\neq|\mu|$). The \emph{Green polynomials} $\{Q_\lambda^\tau(q)\}_{\lambda,\tau\in\calP}$ are defined as 

\begin{equation}Q_\lambda^\tau(q)=\sum_\nu \chi^\nu_\lambda\tilde{K}_{\nu\tau}(q)\end{equation}if $|\lambda|=|\tau|$ and $Q_\lambda^\tau=0$ otherwise.

\subsubsection{Exp and Log}

Let $\Lambda(\x_1,\ldots,\x_k):=
\Lambda(\x_1)\otimes_\Z\cdots\otimes_\Z\Lambda(\x_k)$ be the ring of
functions separately symmetric in each  set $\x_1,\x_2,\ldots,\x_k$
of infinitely many variables.  To ease the notation we will simply write $\Lambda_k$ for the ring
$\Lambda(\x_1,\ldots,\x_k)\otimes_\Z \Q(q,t)$. 

The power series ring $\Lambda_k[[T]]$ is endowed with a natural $\lambda$-ring structure in which the Adams operations are $$\psi_d(f(\x_1,\x_2,\dots,\x_k,q,t;T)):=f(\x_1^d,\x_2^d,\dots,\x_k^d,q^d,t^d;T^d).$$

Let $\Lambda_k[[T]]^+$ be the ideal $T\Lambda_k[[T]]$ of $\Lambda_k[[T]]$.  Define $\Psi: \Lambda_k[[T]]^+\rightarrow\Lambda_k[[T]]^+$ by

$$
\Psi(f):=\sum_{n\geq 1}\frac{\psi_n(f)}{n},
$$
and $\Exp:\Lambda_k[[T]]^+\rightarrow 1+\Lambda_k[[T]]^+$ by 

$$
\Exp(f)=\exp(\Psi(f)).
$$

The inverse $\Psi^{-1}:\Lambda_k[[T]]^+\rightarrow\Lambda_k[[T]]^+$ of $\Psi$ is given by 

$$
\Psi^{-1}(f)=\sum_{n\geq 1}\mu(n)\frac{\psi_n(f)}{n}$$where $\mu$ is the ordinary M\"obius function.

The inverse $\Log:1+\Lambda_k[[T]]\rightarrow\Lambda_k[[T]]$ of $\Exp$ is given by

$$
\Log(f)=\Psi^{-1}(\log(f)).
$$ 

\begin{remark} Let $f=1+\sum_{n\geq 1}f_nT^n\in 1+\Lambda_k[[T]]^+$. If we write

$$
\log\,(f)=\sum_{n\geq 1}\frac{1}{n}U_nT^n,\hspace{1cm}\Log\,(f)=\sum_{n\geq 1}V_nT^n,
$$
then

$$
V_r=\frac{1}{r}\sum_{d | r}\mu(d)\psi_d(U_{r/d}).
$$

\end{remark}

We have the following propositions (details may be found for instance in Mozgovoy \cite{mozgovoy}).

For $g\in \Lambda_k$ and $n\geq 1$ we put 

$$
g_n:=\frac{1}{n}\sum_{d | n}\mu(d)\psi_{\frac{n}{d}}(g).
$$
This is the M\"obius inversion formula of $\psi_n(g)=\sum_{d| n}d\cdot g_d$.

\begin{lemma} Let $g\in \Lambda_k$ and $f_1,f_2\in 1+\Lambda_k[[T]]^+$ such that

$$
\log\,(f_1)=\sum_{d=1}^\infty g_d\cdot\log\,(\psi_d(f_2)).
$$
Then 
$$
\Log\,(f_1)=g\cdot \Log\,(f_2).
$$

\label{moz}
\end{lemma}

\begin{lemma}Assume that $f\in 1+\Lambda_k[[T]]^+$ belongs to
  $\Lambda(\x_1,\ldots,\x_k)\otimes_\Z\Z[q,t]$, then $\Exp(f)\in
  \Lambda(\x_1,\ldots,\x_k)\otimes_\Z\Z[q,t]$.

\label{Exp}\end{lemma}

\subsubsection{Types}

We choose once and for all a total ordering $\geq$ on $\calP$ (e.g. the
lexicographic ordering) and we continue to denote by $\geq$ the total
ordering defined on the set of pairs $\N^*\times\calP^*$ as follows:
If $\lambda\neq\mu$ and $\lambda\geq\mu$, then
$(d,\lambda)\geq(d',\mu)$, and $(d,\lambda)\geq(d',\lambda)$ if $d\geq
d'$. We denote by $\mathbf{T}$ the set of non-increasing sequences
$\omega=(d_1,\omega^1)\geq (d_2,\omega^2)\geq \cdots \geq
(d_r,\omega^r)$, which we will call a {\it type}. To alleviate the
notation we will then omitt the symbol $\geq$ and write simply
$\omega=(d_1,\omega^1)(d_2,\omega^2)\cdots (d_r,\omega^r)$. The {\it
  size} of a type $\omega$ is $|\omega|:=\sum_id_i|\lambda^i|$. We
denote by $\mathbf{T}_n$ the set of types of size $n$. We denote by
$m_{d,\lambda}(\omega)$ the multiplicity of $(d,\lambda)$ in
$\omega$. As with partitions it is sometimes convenient to consider a
type as a collection of integers $m_{d,\lambda}\geq 0 $ indexed by
pairs $(d,\lambda)\in \Nstar\times \; \calP^*$. For a type
$\omega=(d_1,\omega^1)(d_2,\omega^2)\cdots (d_r,\omega^r)$, we put
$n(\omega)=\sum_id_i n(\omega^i)$ and
$[\omega]:=\cup_id_i\cdot\omega^i$.

When considering elements $a_\muhat\in \Lambda_k$ indexed by
multi-partitions $\muhat=(\mu^1,\ldots,\mu^k)\in \calP^k$, we will
always assume that they are homogeneous of degree
$(|\mu^1|,\ldots,|\mu^k|)$ in the set of variables $\x_1,\dots,\x_k$.

Let $\{a_\muhat\}_{\muhat\in\calP^k}$ be a family of symmetric functions in $\Lambda_k$ indexed by multi-partitions.

We extend its definition to a \emph{multi-type} $\omhat=(d_1,\omhat^1)\cdots(d_s,\omhat^s)$ with $\omhat^p\in(\calP_{n_p})^k$, by 

$$
a_\omhat:=\prod_p\psi_{d_p}(A_{\omhat^p}).
$$

For a multi-type $\omhat$ as above, we put
$$
C_\omhat^o:=\begin{cases}\frac{\mu(d)}{d}(-1)^{r-1}\frac{(r-1)!}{\prod_\muhat m_{d,\muhat}(\omhat)!}\,\text{ if } d_1=\cdots=d_r=d.\\0\,\text{ otherwise.}
\end{cases}
$$
where $m_{d,\muhat}(\omhat)$ with $\muhat\in\calP^k$ denotes the multiplicity of $(d,\muhat)$ in $\omhat$.

We have the following lemma (see \cite[\S 2.3.3]{hausel-letellier-villegas} for a proof).

\begin{lemma} 
Let $\{A_\muhat\}_{\muhat\in\calP^k}$ be a family of symmetric functions in $\Lambda_k$ with $A_0=1$. Then
\begin{equation}
\Log\left(\sum_{\muhat\in\calP^k}A_\muhat T^{|\muhat|}\right)=\sum_{\omhat}C_\omhat^oA_\omhat T^{|\omhat|}
\end{equation}
where $\omhat$ runs over multi-types $(d_1,\omhat^1)\cdots(d_s,\omhat^s)$.\label{Log-w}
\end{lemma}

The formal power series $\sum_{n\geq 0}a_nT^n$ with $a_n\in \Lambda_k$
that we will consider in what follows will all have $a_n$ homogeneous
of degree $n$. Hence we will typically  scale the variables of
$\Lambda_k$ by $1/T$ and eliminate $T$ altogether.

Given any family $\{a_\mu\}$ of symmetric functions
indexed by partitions $\mu\in \P$ and a multi-partition  $\muhat \in
\P^k$ as above define
$$
a_\muhat:=a_{\mu^1}(\x_1)\cdots a_{\mu^k}(\x_k).
$$
Let  $\langle\cdot ,\cdot \rangle$ be the Hall pairing on
$\Lambda(\x),$ extend its definition to $\Lambda(\x_1,\ldots,\x_k)$
by setting
\beq \label{extendedhall}
\langle a_1(\x_1)\cdots a_k(\x_k), b_1(\x_1)\cdots b_k(\x_k)\rangle
= \langle a_1, b_1 \rangle \cdots \langle a_k, b_k \rangle,
\eeq
for any $a_1,\ldots,a_k;b_1,\ldots,b_k\in \Lambda(\x)$ and to formal
series by linearity.

\subsubsection{Cauchy identity}\label{Cauchy}

Given a partition $\lambda\in\calP_n$ we
define the genus $g$ {\it hook function}
$\calH_{\lambda}(z,w)$ by
$$
\calH_{\lambda}(z,w):=
\prod_{s\in \lambda}\frac{(z^{2a(s)+1}-w^{2l(s)+1})^{2g}}
{(z^{2a(s)+2}-w^{2l(s)})(z^{2a(s)}-w^{2l(s)+2})},
$$
where the product is over all cells $s$ of $\lambda$ with $a(s)$ and
$l(s)$ its arm and leg length, respectively. For details on the hook
function we refer the reader to \cite{hausel-villegas}.

Recall the specialization (cf. \cite[\S 2.3.5]{hausel-letellier-villegas})  \begin{equation}\calH_\lambda(0,\sqrt{q})=\frac{q^{g\langle\lambda,\lambda\rangle}}{a_\lambda(q)}\label{alambda}\end{equation}where $a_\lambda(q)$ is the cardinality of the centralizer of a unipotent element of $\GL_n(\F_q)$ with Jordan form of type $\lambda$.

It is also not difficult to verify that the Euler specialization of $\calH_\lambda$ is 

\begin{equation}\calH_\lambda(\sqrt{q},1/\sqrt{q})=\left(q^{-\frac{1}{2}\langle\lambda,\lambda\rangle}H_\lambda(q)\right)^{2g-2}.
\label{H-specializ}\end{equation}

We have \begin{equation}\calH_\lambda(z,w)=\calH_{\lambda'}(w,z)\,\,\,{\rm and }\,\,\,\calH_\lambda(-z,-w)=\calH_\lambda(z,w).\label{hook-duality}\end{equation}

Let
$$
\Omega(z,w)=\Omega(\x_1,\dots,\x_k;z,w):=\sum_{\lambda\in \calP} \calH_{\lambda}(z,w)
\prod_{i=1}^k\tilde{H_\lambda}(\x_i;z^2,w^2).
$$By (\ref{Hduality}) and (\ref{hook-duality}) we have
\begin{equation}\Omega(z,w)=\Omega(w,z)\,\,\,{\rm and }\,\,\,\Omega(-z,-w)=\Omega(z,w).\label{Oduality}\end{equation}

For $\muhat=(\mu^1,\cdots,\mu^k)\in\P^k$, we let

\begin{equation}\H_\muhat(z,w):=(z^2-1)(1-w^2)\left\langle\Log\,\Omega(z,w),h_\muhat\right\rangle.
 \label{H}\end{equation}

By (\ref{Oduality}) we have the symmetries

\begin{equation}\H_\muhat(z,w)=\H_\muhat(w,z) \,\,\,{\rm and}\,\,\,\H_\muhat(-z,-w)=\H_\muhat(z,w). \label{HHduality}\end{equation}

We may recover $\Omega(z,w)$ from the $\H_\muhat(z,w)$'s by the formula:

\beq
\Omega(z,w)=\Exp\left(\sum_{\muhat\in\calP^k}\frac{\H_\muhat(z,w)}{(z^2-1)(1-w^2)}m_\muhat\right).
\label{exp}\eeq

From Formula (\ref{eulermac}) and Formula (\ref{H-specializ}) we have:

\begin{lemma} With the specialization $y_i=q^{i-1}$, 

$$\Omega\left(\sqrt{q},\frac{1}{\sqrt{q}}\right)=\sum_{\lambda\in\calP}q^{(1-q)|\lambda|}\left(q^{-n(\lambda)}H_\lambda(q)\right)^{2g+k-2}\prod_{i=1}^ks_\lambda(\x_i\y).$$
 
\label{specializ}\end{lemma}

\begin{conjecture}The rational function $\H_\muhat(z,w)$ is a polynomial with integer coefficients. It has degree 

$$d_\muhat:=n^2(2g-2+k)-\sum_{i,j}(\mu^i_j)^2+2$$ in each variable and the coefficients of $\H_\muhat(-z,w)$ are non-negative.
\label{conjH}\end{conjecture}

The function $\H_\muhat(z,w)$ is computed in many cases in \cite[\S 1.5]{hausel-letellier-villegas}.

\subsection{Characters and Fourier transforms}\label{finite-groups}

\subsubsection{Characters of finite general linear groups}\label{charGL}For a finite group $H$ let us denote by ${\rm Mod}_H$ the category of finite dimensional $\C[H]$ left modules. Let $K$ be an other finite group. By an \emph{$H$-module-$K$} we mean a finite dimensional $\C$-vector space $M$ endowed with a left action of $H$ and with a right action of $K$ which commute together. Such a module $M$ defines a functor $R_K^H:{\rm Mod}_K\rightarrow{\rm Mod}_H$ by $V\mapsto M\otimes_{\C[K]} V$. Let $\C(H)$ denotes the $\C$-vector space of all functions $H\rightarrow\C$ which are constant on conjugacy classes.  We continue to denote by $R_K^H$ the $\C$-linear map $\C(K)\rightarrow\C(H)$ induced by the functor $R_K^H$ (we first define it on irreducible characters and then extend it by linearity to the whole $\C(K)$). Then for any $f\in\C(K)$, we have

\begin{equation}R_K^H(f)(g)=|K|^{-1}\sum_{k\in K}{\rm Trace}\,\left((g,k^{-1})\,|\, M\right)f(k).\label{bimod}
\end{equation}

Let $G=\GL_n(\F_q)$ with $\F_q$ a finite field. Fix a partition $\lambda=(\lambda_1,\dots,\lambda_r)$ of $n$ and let $\calF_\lambda=\calF_\lambda(\F_q)$ be the variety of partial flags of $\F_q$-vector spaces $$\{0\}=E^r\subset E^{r-1}\subset\cdots\subset E^1\subset E^0=(\F_q)^n$$such that ${\rm dim}(E^{i-1}/E^i)=\lambda_i$.

Let $G$ acts on $\calF_\lambda$ in the natural way. Fix an element $$X_o=\left(\{0\}=E^r\subset E^{r-1}\subset\cdots\subset E^1\subset E^0=(\F_q)^n\right)\in\calF_\lambda$$ and denote by $P_\lambda$ the stabilizer of $X_o$ in $G$ and by $U_\lambda$ the subgroup of elements $g\in P_\lambda$ which induces the identity on $E^i/E^{i+1}$ for all $i=0,1,\dots,r-1$. 

Put $L_\lambda:=\GL_{\lambda_r}(\F_q)\times\cdots\times\GL_{\lambda_1}(\F_q)$. Recall that $U_\lambda$ is a normal subgroup of $P_\lambda$ and that  $P_\lambda=L_\lambda\ltimes U_\lambda$.

Denote by $\C[G/U_\lambda]$ the $\C$-vector space generated by the finite set $G/U_\lambda=\{gU_\lambda\,|\, g\in G\}$. The group $L_\lambda$ (resp. $G$) acts on $\C[G/U_\lambda]$ as $(gU_\lambda)\cdot l=glU_\lambda$ (resp. as $g\cdot (hU_\lambda)=ghU_\lambda$). These two actions make $\C[G/U_\lambda]$ into a $G$-module-$L_\lambda$. The associated functor $R_{L_\lambda}^G:{\rm Mod}_{L_\lambda}\rightarrow {\rm Mod}_G$ is the so-called \emph{Harish-Chandra functor}.

We have the following well-known lemma.

\begin{lemma} We denote by $1$ the identity character of $L_\lambda$. Then for all $g\in G$, we have

$$R_{L_\lambda}^G(1)(g)=\#\{X\in\calF_\lambda\,|\, g\cdot X=X\}.$$
\label{R=F}\end{lemma}

\begin{proof} By Formula (\ref{bimod}) we have

\begin{align*} R_{L_\lambda}^G(1)(g)&=|L_\lambda|^{-1}\sum_{k\in L_\lambda}\#\{hU_\lambda\,|\, ghU_\lambda=hkU_\lambda\}\\
&=|L_\lambda|^{-1}\sum_{k\in L_\lambda}\#\{hU_\lambda\,|\, gh\in hkU_\lambda\}\\
&=|L_\lambda|^{-1}\#\{hU_\lambda\,|\, gh\in hP_\lambda\}\\
&=\#\{hP_\lambda\,|\, ghP_\lambda=hP_\lambda\}.
 \end{align*}
We deduce the lemma from last equality by noticing that the map $G\rightarrow\calF_\lambda$, $g\mapsto g\cdot X_o$ induces a bijection $G/P_\lambda\rightarrow\calF_\lambda$.
\end{proof}

We now recall the definition of the type of a conjugacy class $C$ of $G$ (cf. \cite[4.1]{hausel-letellier-villegas}). The Frobenius $f:\F\rightarrow\F$, $x\mapsto x^q$ acts on the set of eigenvalues of $C$. Let us write the set of eigenvalues of $C$ as a disjoint union 

$$\{\gamma_1,\gamma_1^q,\dots\}\coprod\{\gamma_2,\gamma_2^q,\dots\}\coprod\cdots\coprod\{\gamma_r,\gamma_r^q,\dots\}$$of $\langle f\rangle$-orbits, and let $m_i$ be the multiplicity of $\gamma_i$. The unipotent part of an element of $C$ defines a unique partition $\omega^i$ of $m_i$. Re-ordering if necessary we may assume that $(d_1,\omega^1)\geq (d_2,\omega^2)\geq\cdots\geq (d_r,\omega^r)$. We then call $\omega=(d_1,\omega^1)\cdots(d_r,\omega^r)\in{\bf T}_n$ the \emph{type} of $C$.

Put $T:=L_{(1,1,\dots,1)}$. It is the subgroup of diagonal matrices of
$G$. The decomposition of $R_T^{L_\lambda}(1)$ as a sum of irreducible
characters reads $$R_T^{L_\lambda}(1)=\sum_{\chi\in {\rm
    Irr}(W_{L_\lambda})}\chi(1)\cdot \calU_\chi,$$where
$W_{L_\lambda}:=N_{L_\lambda}(T)/T$ is the Weyl group of $L_\lambda$. We call the
irreducible characters $\{\calU_\chi\}_\chi$ the \emph{unipotent}
characters of $L_\lambda$. The character  $\calU_1$ is the trivial
character of $L_\lambda$. Since $W_{L_\lambda}\simeq
S_{\lambda_1}\times\cdots\times S_{\lambda_r}$, the irreducible
characters of $W_{L_\lambda}$ are
$\chi^\tau:=\chi^{\tau^1}\cdots\chi^{\tau^r}$ where $\tau$
runs over the set of types $\tau=\{(1,\tau^i)\}_{i=1,\dots,r}$
with $\tau^i$ a partition of $\lambda_i$. We denote by
$\calU_\tau$ the unipotent character of $L_\lambda$ corresponding to
such a type $\tau$.

\begin{theorem} Let $\calU_\tau$ be a unipotent character of $L_\lambda$ and let $C$ be a conjugacy class of type $\omega$. Then

$$R_{L_\lambda}^G(\calU_\tau)(C)=\left\langle \tilde{H}_\omega(\x,q),s_\tau(\x)\right\rangle.$$

\label{Rtau}\end{theorem}

\begin{proof} The proof  is contained in \cite{hausel-letellier-villegas} although the formula is not explicitely written there. For the convenience of the reader we now explain how to extract the proof from \cite{hausel-letellier-villegas}. For $w\in W_\lambda$, we denote by $R_{T_w}^G(1)$ the corresponding \emph{Deligne-Lusztig character} of $G$. Its construction is outlined in \cite[2.6.4]{hausel-letellier-villegas}. The character $\calU_\tau$ of $L_\lambda$ decomposes as,

$$\calU_\tau=|W_\lambda|^{-1}\sum_{w\in W_\lambda}\chi^\tau_w\cdot R_{T_w}^{L_\lambda}(1)$$where $\chi^\tau_w$ denotes the value of $\chi^\tau$ at $w$. Applying the Harish-Chandra induction $R_{L_\lambda}^G$ to both side and using the transitivity of induction we find that

$$R_{L_\lambda}^G(\calU_\tau)=|W_\lambda|^{-1}\sum_{w\in W_\lambda}\chi^\tau_w\cdot R_{T_w}^G(1).$$We are now in position to use the calculation in \cite{hausel-letellier-villegas}. Notice that the right handside of the above formula is the right hand side of the first formula displayed in the proof of \cite[Theorem 4.3.1]{hausel-letellier-villegas} with $(M,\theta^{T_w},\tilde{\varphi})=(L_\lambda,1,\chi^\tau)$ and so the same calculation to get  \cite[(4.3.2)]{hausel-letellier-villegas} together with \cite[(4.3.3)]{hausel-letellier-villegas} gives in our case

$$R_{L_\lambda}^G(\calU_\tau)(C)=\sum_\alpha z_\alpha^{-1}\chi^\tau_\alpha\sum_{\{\beta\,|\,[\beta]=[\alpha]\}}Q_\beta^\omega(q)z_{[\alpha]}z_\beta^{-1}$$where the notation are those of \cite[§ 4.3]{hausel-letellier-villegas}. We now apply \cite[Lemma 2.3.5]{hausel-letellier-villegas} to get 

$$R_{L_\lambda}^G(\calU_\tau)(C)=\left\langle \tilde{H}_\omega(\x;q),s_\tau(\x)\right\rangle.$$
\end{proof}

If $\alpha$ is the type $(1,(\lambda_1))\cdots(1,(\lambda_r))$, then $s_\alpha(\x)=h_\lambda(\x)$. Hence we have:

\begin{corollary} If $C$ is a conjugacy class of $G$ type $\omega$, then 

$$R_{L_\lambda}^G(1)(C)=\left\langle \tilde{H}_\omega(\x,q),h_\lambda(\x)\right\rangle.$$
 \label{R}

\end{corollary}

\begin{corollary} Put $\calF^\#_{\lambda,\omega}(q):=\#\{X\in\calF_\lambda\,|\, g\cdot X=X\}$ where $g\in G$ is an element in a conjugacy class of type $\omega$. Then 

$$\tilde{H}_\omega(\x,q)=\sum_\lambda \calF^\#_{\lambda,\omega}(q)m_\lambda(\x).$$

\end{corollary}

\begin{proof} It follows from Lemma \ref{R=F} and Corollary \ref{R}.\end{proof}

We now recall how to construct from a partition $\lambda=(\lambda_1,\dots,\lambda_r)$ of $n$ a certain  family of irreducible characters of $G$. Choose $r$ distinct linear character $\alpha_1,\dots,\alpha_r$ of $\F_q^{\times}$. This defines for each $i$ a linear characters $\tilde{\alpha}_i:\GL_{\lambda_i}(\F_q)\rightarrow\C^{\times}$, $g\mapsto \alpha_i\left({\rm det}(g)\right)$, and hence a linear character $\tilde{\alphahat}: L_\lambda\rightarrow\C^{\times}$, $(g_i)\mapsto \tilde{\alpha}_r(g_r)\cdots\tilde{\alpha}_1(g_1)$. This linear character has the following property: for an element $g\in N_G(L_\lambda)$, we have $\tilde{\alpha}(g^{-1}lg)=\tilde{\alpha}(l)$ for all $l\in L_\lambda$ if and only if $g\in L_\lambda$. A linear character of $L_\lambda$ which satifies this property is called a \emph{regular} character of $L_\lambda$. 

It is a well-known fact that $R_{L_\lambda}^G(\tilde{\alphahat})$ is an irreducible character of $G$. Note that the irreducible characters of $G$ are not all obtained in this way (see \cite{LSr} for the complete description of the irreducible characters of $G$ in terms of Deligne-Luzstig induction).

We now recall the definition of generic tuples of irreducible characters (cf. \cite[Definition 4.2.2]{hausel-letellier-villegas}). Since in this paper we are only considering irreducible characters of the form $R_{L_\lambda}^G(\tilde{\alphahat})$, the definition given in \cite[Definition 4.2.2]{hausel-letellier-villegas} simplifies.

\begin{definition}Consider irreducible characters  $R_{L_{\lambda^1}}^G(\tilde{\alphahat}_1),\dots,R_{L_{\lambda^k}}^G(\tilde{\alphahat}_k)$ of $G$ as above for a multi-partition $\lambdahat=(\lambda^1,\dots,\lambda^k)\in(\calP_n)^k$. Let $T$ be the subgroup of $G$ of diagonal matrices. Note that $T\subset L_\lambda$ for all partition $\lambda$, and so $T$ contains the center $Z_\lambda$ of any $L_\lambda$. Consider the linear character $\alphahat=\left(\tilde{\alphahat}_1|_T\right)\cdots\left(\tilde{\alphahat}_k|_T\right)$ of $T$. Then we say that the tuple $\left(R_{L_{\lambda^1}}^G(\tilde{\alphahat}_1),\dots,R_{L_{\lambda^k}}^G(\tilde{\alphahat}_k)\right)$ is \emph{generic} if the restriction $\alphahat|_{Z_\lambda}$ of $\alphahat$ to any subtori $Z_\lambda$, with $\lambda\in\calP_n-\{(n)\}$, is non-trivial and if $\alphahat|_{Z_{(n)}}$ is trivial (the center $Z_{(n)}\simeq \F_q^{\times}$ consists of scalar matrices $a.I_n$).\end{definition}

We can show as for conjugacy classes \cite[Lemma 2.1.2]{hausel-letellier-villegas} that if the characteristic $p$ of $\F_q$ and $q$ are sufficiently large, generic tuples of irreducible characters of a given type $\lambdahat$ always exist.

Put $\mathfrak{g}:=\gl_n(\F_q)$. For $X\in\mathfrak{g}$, put $$\Lambda^1(X):=\#\{Y\in\mathfrak{g}\,|\, [X,Y]=0\}.$$

The restriction $\Lambda^1:G\rightarrow\C$ of $\Lambda^1$ to $G\subset\mathfrak{g}$ is the character of the representation $G\rightarrow\GL\left(\C[\mathfrak{g}]\right)$ induced by the conjugation action of $G$ on $\mathfrak{g}$. Fix a non-negative integer $g$ and  put $\Lambda:=(\Lambda^1)^{\otimes g}$.

For a multi-partition $\muhat=(\mu^1,\dots,\mu^k)\in(\calP_n)^k$ and a generic tuple $\left(R_{L_{\mu^1}}^G(\tilde{\alphahat}_1),\dots,R_{L_{\mu^k}}^G(\tilde{\alphahat}_k)\right)$ of irreducible characters we put $$R_\muhat:=R_{L_{\mu^1}}^G(\tilde{\alphahat}_1)\otimes\cdots\otimes R_{L_{\mu^k}}^G(\tilde{\alphahat}_k).$$

For two class functions $f,g\in\C(G)$, we define $$\langle f,g\rangle:=|G|^{-1}\sum_{h\in G}f(h)\overline{g(h)}.$$

We have the following theorem \cite[Theorem 1.4.1]{hausel-letellier-villegas}.

\begin{theorem} We have $$\left\langle \Lambda\otimes R_\muhat,1\right\rangle=\H_\muhat\left(0,\sqrt{q}\right)$$where $\H_\muhat(z,w)$ is the function defined in \S \ref{Cauchy}.

\label{multi}\end{theorem}

\begin{corollary} The multiplicity $\left\langle \Lambda\otimes
    R_\muhat,1\right\rangle$ depends only on $\muhat$ and not on the
  choice of linear characters
  $(\tilde{\alphahat}_1,\dots,\tilde{\alphahat}_k)$.
 \end{corollary}

\subsubsection{Fourier transforms}\label{Fourier}

Let ${\rm Fun}(\mathfrak{g})$ be the $\C$-vector space of all functions $\mathfrak{g}\rightarrow\C$ and by $\C(\mathfrak{g})$ the subspace  of functions $\mathfrak{g}\rightarrow\C$ which are contant on $G$-orbits of $\mathfrak{g}$ for the conjugation action of $G$ on $\mathfrak{g}$. 

Let $\Psi:\F_q\rightarrow\C^{\times}$ be a non-trivial additive character and consider the trace pairing ${\rm Tr}:\mathfrak{g}\times\mathfrak{g}\rightarrow\C^{\times}$. Define the Fourier transform $\calF^\mathfrak{g}:{\rm Fun}(\mathfrak{g})\rightarrow{\rm Fun}(\mathfrak{g})$ by the formula 

$$\calF^\mathfrak{g}(f)(x)=\sum_{y\in\mathfrak{g}}\Psi\left({\rm Tr}\,(xy)\right)f(y)$$for all $f\in{\rm Fun}(\mathfrak{g})$ and $x\in\mathfrak{g}$.

The Fourier transform satisfies the following easy property.
\begin{proposition}For any $f\in{\rm Fun}(\mathfrak{g})$ we have:
 
$$|\mathfrak{g}|\cdot f(0)=\sum_{x\in\mathfrak{g}}\calF^\mathfrak{g}(f)(x).$$
\label{fourprop1}\end{proposition}

Let $*$ be the convolution product on ${\rm Fun}(\mathfrak{g})$ defined by $$(f*g)(a)=\sum_{x+y=a}f(x)g(y)$$for any two functions $f,g\in{\rm Fun}(\mathfrak{g})$.

Recall that 

\begin{equation}\calF^\mathfrak{g}(f*g)=\calF^\mathfrak{g}(f)\cdot\calF^\mathfrak{g}(g).\label{conv}\end{equation}

For a partition $\lambda$ of $n$, let $\mathfrak{p}_\lambda$, $\mathfrak{l}_\lambda$, $\mathfrak{u}_\lambda$ be the Lie sub-algebras of $\mathfrak{g}$ corresponding respectively to the subgroups $P_\lambda$, $L_\lambda$, $U_\lambda$ defined in \S \ref{finite-groups}, namely $\mathfrak{l}_\lambda=\bigoplus_i\gl_{\lambda_i}(\F_q)$, $\mathfrak{p}_\lambda$ is the parabolic sub-algebra of $\mathfrak{g}$ having $\mathfrak{l}_\lambda$ as a Levi sub-algebra and containing the upper triangular matrices. We have $\mathfrak{p}_\lambda=\mathfrak{l}_\lambda\oplus\mathfrak{u}_\lambda$.

Define the two functions $R_{\mathfrak{l}_\lambda}^{\mathfrak{g}}(1), Q_{\mathfrak{l}_\lambda}^\mathfrak{g} \in\C(\mathfrak{g})$ by

\begin{align*}&R_{\mathfrak{l}_\lambda}^{\mathfrak{g}}(1)(x)=|P_\lambda|^{-1}\#\{g\in G\,|\, g^{-1}xg\in \mathfrak{p}_\lambda\},\\&Q_{\mathfrak{l}_\lambda}^{\mathfrak{g}}(x)=|P_\lambda|^{-1}\#\{g\in G\,|\, g^{-1}xg\in \mathfrak{u}_\lambda\}.\end{align*}

We define the type of a $G$-orbit of $\mathfrak{g}$ similarly as in the group setting (see above Corollary \ref{R}). The types of the $G$-orbits of $\mathfrak{g}$ are then also parameterized by $\mathbf{T}_n$.

\begin{remark}From Lemma \ref{R=F}, we see that $R_{L_\lambda}^G(1)(x)=|P_\lambda|^{-1}\#\{g\in G\,|\, g^{-1}xg\in P_\lambda\}$, hence $R_{\mathfrak{l}_\lambda}^{\mathfrak{g}}(1)$ is the Lie algebra analogue of $R_{L_\lambda}^G(1)$ and the two functions take the same values on elements of same type.\label{Rl=RL}\end{remark}

\begin{proposition} We have $$\calF^\mathfrak{g}\left(Q_{\mathfrak{l}_\lambda}^\mathfrak{g}\right)=q^{\frac{1}{2}(n^2-\sum_i\lambda_i^2)}R_{\mathfrak{l}_\lambda}^\mathfrak{g}(1).$$
\label{fourprop2}\end{proposition}

\begin{proof} Consider the $\C$-linear map $R_{\mathfrak{l}_\lambda}^\mathfrak{g}:\C(\mathfrak{l}_\lambda)\rightarrow\C(\mathfrak{g})$ defined by 

$$R_{\mathfrak{l}_\lambda}^\mathfrak{g}(f)(x)=|P_\lambda|^{-1}\sum_{\{g\in G\,|\, g^{-1}xg\in \mathfrak{p}_\lambda\}}f(\pi(g^{-1}xg))$$where $\pi:\mathfrak{p}_\lambda\rightarrow\mathfrak{l}_\lambda$ is the canonical projection. Then it is easy to see that $Q_{\mathfrak{l}_\lambda}^\mathfrak{g}=R_{\mathfrak{l}_\lambda}^\mathfrak{g}(1_0)$ where $1_0\in\C(\mathfrak{l}_\lambda)$ is the characteristic function of $0\in\mathfrak{l}_\lambda$, i.e., $1_0(x)=1$ if $x=0$ and $1_0(x)=0$ otherwise. The result follows from the easy fact that $\calF^{\mathfrak{l}_\lambda}(1_0)$ is the identity function $1$ on $\mathfrak{l}_\lambda$ and the fact (see Lehrer \cite{lehrer}) that $$\calF^\mathfrak{g}\circ R_{\mathfrak{l}_\lambda}^\mathfrak{g}=q^{\frac{1}{2}(n^2-\sum_i\lambda_i^2)}R_{\mathfrak{l}_\lambda}^\mathfrak{g}\circ\calF^{\mathfrak{l}_\lambda}.$$
\end{proof}

\begin{remark}For $x\in\mathfrak{g}$, denote by $1_x\in{\rm Fun}(\mathfrak{g})$ the characteristic function of $x$ that takes the value $1$ at $x$ and the value $0$ elsewhere. Note that $\calF^\mathfrak{g}(1_x)$ is the linear character $\mathfrak{g}\rightarrow\C$, $t\mapsto \Psi({\rm Tr}\,(xt))$ of the abelian group $(\mathfrak{g},+)$. Hence if $f:\mathfrak{g}\rightarrow\C$ is a function which takes integer values, then $\calF^\mathfrak{g}(f)$ is a character (not necessarily irreducible) of $(\mathfrak{g},+)$. Since the Green functions $Q_{\mathfrak{l}_\lambda}^\mathfrak{g}$ take integer values, by Proposition \ref{fourprop2} the function $q^{\frac{1}{2}(n^2-\sum_i\lambda_i^2)}R_{\mathfrak{l}_\lambda}^\mathfrak{g}(1)$ is a character of $(\mathfrak{g},+)$.

\label{charab}\end{remark}

\section{Absolutely indecomposable representations}

\subsection{Generalities on quiver representations}\label{genquiv}

Let $\Gamma$ be a finite quiver, $I$ be the set of its vertices and
let $\Omega$ be the set of its arrows. For $\gamma\in\Omega$, we
denote by $h(\gamma),t(\gamma)\in I$ the head and the tail of
$\gamma$. A \emph{dimension vector} of $\Gamma$ is a collection of
non-negative integers $\v=\{v_i\}_{i\in I}$ and a
\emph{representation} $\varphi$ of $\Gamma$ of dimension $\v$ over a
field $\K$ is a collection of $\K$-linear maps
$\varphihat=\{\varphi_\gamma:V_{t(\gamma)}\rightarrow
V_{h(\gamma)}\}_{\gamma\in\Omega}$ with ${\rm dim}\, V_i=v_i$. Let
${\rm Rep}_{\Gamma,\v}(\K)$ be the $\K$-vector space of all
representations of $\Gamma$ of dimension $\v$ over $\K$. If
$\varphihat\in {\rm Rep}_{\Gamma,\v}(\K)$, $\varphihat'\in {\rm
  Rep}_{\Gamma,\v'}(\K)$, then a morphism
$f:\varphihat\rightarrow\varphihat'$ is a collection of $\K$-linear
maps $f_i: V_i\rightarrow V_i'$, $i\in I$ such that for all
$\gamma\in\Omega$, we have $f_{h(\gamma)}\circ
\varphi_\gamma=\varphi'_\gamma\circ f_{t(\gamma)}$.

We define in the obvious way direct sums
$\varphihat\,\oplus\,\varphihat'\in {\rm Rep}_\K(\Gamma,\v+\v')$ of
representations.  A representation of $\Gamma$ is said to be
\emph{indecomposable} over $\K$ if it is not isomorphic to a direct
sum of two non-zero representations of $\Gamma$. If an indecomposable
representation of $\Gamma$ remains indecomposable over any finite
extension of $\K$, we say that it is \emph{absolutely indecomposable}.
Denote by ${\rm M}_{\Gamma,\v}(\K)$ be the set of isomorphism classes
of ${\rm Rep}_{\Gamma,\v}(\K)$ and by ${\rm A}_{\Gamma,\v}(\K)$ the
subset of absolutely indecomposable representations of ${\rm
  Rep}_{\Gamma,\v}(\K)$.

By a theorem of Kac there exists a polynomial
$A_{\Gamma,\v}(T)\in \Z[T]$ such that for any finite field with $q$ elements
$A_{\Gamma,\v}(q)=\#{\rm A}_{\Gamma,\v}(\F_q)$. We
call $A_{\Gamma,\v}$ the \emph{$A$-polynomial} of $(\Gamma,\v)$.

Let $\Phi(\Gamma)\subset \Z^I$ be the root system associated with the
quiver $\Gamma$ following Kac \cite{kacconj} and
let $\Phi(\Gamma)^+\subset \left(\N\right)^I$ be the subset of positive
roots.  Let ${\bf C}=(c_{ij})_{i,j}$ be the Cartan matrix of $\Gamma$, namely
$$
c_{ij}=\begin{cases} 2-2(\text{the number of edges joining $i$ to
    itself})\hspace{.2cm}\text{if }i=j\\ 
  - (\text{the number of edges joining $i$ to
    $j$})\hspace{1.2cm}\text{ otherwise}.
         \end{cases}
$$

Then we have the following well-known theorem (see Kac
\cite{kacconj}).

\begin{theorem} $A_{\Gamma,\v}(q)\neq 0$ if and only if
  $\v\in\Phi(\Gamma)^+$; $A_{\Gamma,\v}(q)=1$ if and only if $\v$ is a
  real root.  The polynomial $A_{\Gamma,\v}$, if non-zero, is monic of
  degree $2-{^t}\v{\bf C}\v$.
\label{kactheo} 
\end{theorem}

We have the following conjecture due to Kac \cite{kacconj}.

\begin{conjecture} 
  The polynomial $A_{\Gamma,\v}(T)$ has non-negative coefficients.
\label{kac-conj}
\end{conjecture}

We will say that a dimension vector $\v$ is \emph{indivisible} if
${\rm gcd}\,\{v_i\}_{i\in I}=1$.  Conjecture \ref{kac-conj} was proved
by Crawley-Boevey and van den Bergh \cite{crawley-boevey-etal} when
the dimension vector $\v$ is indivisible. This was achieved by giving
a cohomological interpretation of $A_{\Gamma,\v}(q)$. A more recent work
by Mozgovoy \cite{mozgovoy2} proves Conjecture \ref{kac-conj} for any dimension vector for quivers
with at least one loop at each vertex. His proof is accomplished via work
of Kontsevich-Soibelman \cite{kontsevich-soibelman} and Efimov \cite{efimov}
on motivic Donaldson-Thomas invariants associated to quivers. 

By Kac\cite{kacconj}, there exists a polynomial
$M_{\Gamma,\v}(q)\in\Q[T]$ such that $M_{\Gamma,\v}(q):=\#{\rm
  M}_{\Gamma,\v}(\F_q)$ for any finite field $\F_q$.  The following
formula is a reformation of Hua's formula \cite{hua}.
\begin{theorem}
We have 
$$
\Log\left(\sum_{\v\in(\N)^I}M_{\Gamma,\v}(q)X^{\v}\right)
=\sum_{\v\in(\N)^I-\{0\}}A_{\Gamma,\v}(q)X^{\v},
$$
where $X^{\v}$ is the monomial $\prod_{i\in I}X_i^{v_i}$ for some
independent commuting variables $\{X_i\}_{i\in I}$.
\label{hua}.
\end{theorem}

Since $A_{\Gamma,\v}(q)\in\Z[q]$, we see by Theorem \ref{hua} and
Lemma \ref{Exp}, that $M_{\Gamma,\v}(q)$ also has integer
coefficients.

\subsection{Comet-shaped quivers}\label{comet}

Fix strictly positive integers $g,k,s_1,\dots,s_k$ and consider the following (comet-shaped) quiver $\Gamma$  with $g$ loops on the central vertex and with set of vertices $I=\{0\}\cup\left\{[i,j]\,|\, i=1,\dots,k\,;\, j=1,\dots,s_i\right\}$.

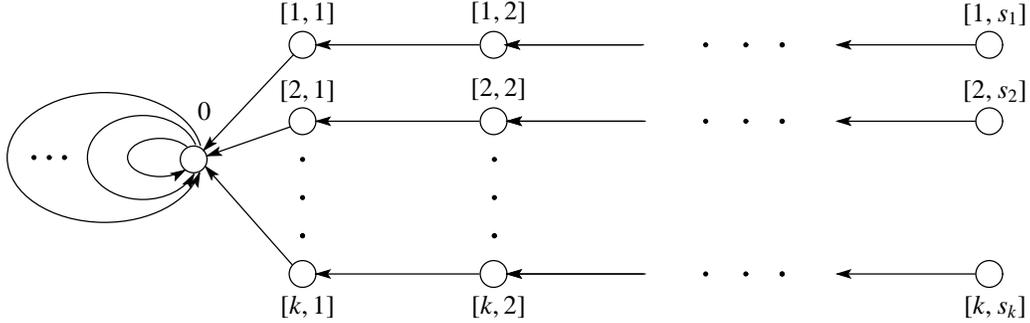
\begin{center}
\unitlength 0.1in
\begin{picture}( 52.1000, 15.4500)(  4.0000,-17.0000)
%
\special{pn 8}%
\special{ar 1376 1010 70 70  0.0000000 6.2831853}%
%
\special{pn 8}%
\special{ar 1946 410 70 70  0.0000000 6.2831853}%
%
\special{pn 8}%
\special{ar 2946 410 70 70  0.0000000 6.2831853}%
%
\special{pn 8}%
\special{ar 5540 410 70 70  0.0000000 6.2831853}%
%
\special{pn 8}%
\special{ar 1946 810 70 70  0.0000000 6.2831853}%
%
\special{pn 8}%
\special{ar 2946 810 70 70  0.0000000 6.2831853}%
%
\special{pn 8}%
\special{ar 5540 810 70 70  0.0000000 6.2831853}%
%
\special{pn 8}%
\special{ar 1946 1610 70 70  0.0000000 6.2831853}%
%
\special{pn 8}%
\special{ar 2946 1610 70 70  0.0000000 6.2831853}%
%
\special{pn 8}%
\special{ar 5540 1610 70 70  0.0000000 6.2831853}%
%
\special{pn 8}%
\special{pa 1890 1560}%
\special{pa 1440 1050}%
\special{fp}%
\special{sh 1}%
\special{pa 1440 1050}%
\special{pa 1470 1114}%
\special{pa 1476 1090}%
\special{pa 1500 1088}%
\special{pa 1440 1050}%
\special{fp}%
%
\special{pn 8}%
\special{pa 2870 410}%
\special{pa 2020 410}%
\special{fp}%
\special{sh 1}%
\special{pa 2020 410}%
\special{pa 2088 430}%
\special{pa 2074 410}%
\special{pa 2088 390}%
\special{pa 2020 410}%
\special{fp}%
%
\special{pn 8}%
\special{pa 3720 410}%
\special{pa 3010 410}%
\special{fp}%
\special{sh 1}%
\special{pa 3010 410}%
\special{pa 3078 430}%
\special{pa 3064 410}%
\special{pa 3078 390}%
\special{pa 3010 410}%
\special{fp}%
\special{pa 3730 410}%
\special{pa 3010 410}%
\special{fp}%
\special{sh 1}%
\special{pa 3010 410}%
\special{pa 3078 430}%
\special{pa 3064 410}%
\special{pa 3078 390}%
\special{pa 3010 410}%
\special{fp}%
%
\special{pn 8}%
\special{pa 2870 810}%
\special{pa 2020 810}%
\special{fp}%
\special{sh 1}%
\special{pa 2020 810}%
\special{pa 2088 830}%
\special{pa 2074 810}%
\special{pa 2088 790}%
\special{pa 2020 810}%
\special{fp}%
%
\special{pn 8}%
\special{pa 2870 1610}%
\special{pa 2020 1610}%
\special{fp}%
\special{sh 1}%
\special{pa 2020 1610}%
\special{pa 2088 1630}%
\special{pa 2074 1610}%
\special{pa 2088 1590}%
\special{pa 2020 1610}%
\special{fp}%
%
\special{pn 8}%
\special{pa 3730 810}%
\special{pa 3020 810}%
\special{fp}%
\special{sh 1}%
\special{pa 3020 810}%
\special{pa 3088 830}%
\special{pa 3074 810}%
\special{pa 3088 790}%
\special{pa 3020 810}%
\special{fp}%
\special{pa 3740 810}%
\special{pa 3020 810}%
\special{fp}%
\special{sh 1}%
\special{pa 3020 810}%
\special{pa 3088 830}%
\special{pa 3074 810}%
\special{pa 3088 790}%
\special{pa 3020 810}%
\special{fp}%
%
\special{pn 8}%
\special{pa 3730 1610}%
\special{pa 3020 1610}%
\special{fp}%
\special{sh 1}%
\special{pa 3020 1610}%
\special{pa 3088 1630}%
\special{pa 3074 1610}%
\special{pa 3088 1590}%
\special{pa 3020 1610}%
\special{fp}%
\special{pa 3740 1610}%
\special{pa 3020 1610}%
\special{fp}%
\special{sh 1}%
\special{pa 3020 1610}%
\special{pa 3088 1630}%
\special{pa 3074 1610}%
\special{pa 3088 1590}%
\special{pa 3020 1610}%
\special{fp}%
%
\special{pn 8}%
\special{pa 5466 410}%
\special{pa 4746 410}%
\special{fp}%
\special{sh 1}%
\special{pa 4746 410}%
\special{pa 4812 430}%
\special{pa 4798 410}%
\special{pa 4812 390}%
\special{pa 4746 410}%
\special{fp}%
%
\special{pn 8}%
\special{pa 5466 810}%
\special{pa 4746 810}%
\special{fp}%
\special{sh 1}%
\special{pa 4746 810}%
\special{pa 4812 830}%
\special{pa 4798 810}%
\special{pa 4812 790}%
\special{pa 4746 810}%
\special{fp}%
%
\special{pn 8}%
\special{pa 5466 1610}%
\special{pa 4746 1610}%
\special{fp}%
\special{sh 1}%
\special{pa 4746 1610}%
\special{pa 4812 1630}%
\special{pa 4798 1610}%
\special{pa 4812 1590}%
\special{pa 4746 1610}%
\special{fp}%
%
\special{pn 8}%
\special{pa 1880 840}%
\special{pa 1450 990}%
\special{fp}%
\special{sh 1}%
\special{pa 1450 990}%
\special{pa 1520 988}%
\special{pa 1500 972}%
\special{pa 1506 950}%
\special{pa 1450 990}%
\special{fp}%
%
\special{pn 8}%
\special{pa 1900 460}%
\special{pa 1430 960}%
\special{fp}%
\special{sh 1}%
\special{pa 1430 960}%
\special{pa 1490 926}%
\special{pa 1468 922}%
\special{pa 1462 898}%
\special{pa 1430 960}%
\special{fp}%
%
\special{pn 8}%
\special{sh 1}%
\special{ar 1946 1010 10 10 0  6.28318530717959E+0000}%
\special{sh 1}%
\special{ar 1946 1210 10 10 0  6.28318530717959E+0000}%
\special{sh 1}%
\special{ar 1946 1410 10 10 0  6.28318530717959E+0000}%
\special{sh 1}%
\special{ar 1946 1410 10 10 0  6.28318530717959E+0000}%
%
\special{pn 8}%
\special{sh 1}%
\special{ar 4056 410 10 10 0  6.28318530717959E+0000}%
\special{sh 1}%
\special{ar 4266 410 10 10 0  6.28318530717959E+0000}%
\special{sh 1}%
\special{ar 4456 410 10 10 0  6.28318530717959E+0000}%
\special{sh 1}%
\special{ar 4456 410 10 10 0  6.28318530717959E+0000}%
%
\special{pn 8}%
\special{sh 1}%
\special{ar 4056 810 10 10 0  6.28318530717959E+0000}%
\special{sh 1}%
\special{ar 4266 810 10 10 0  6.28318530717959E+0000}%
\special{sh 1}%
\special{ar 4456 810 10 10 0  6.28318530717959E+0000}%
\special{sh 1}%
\special{ar 4456 810 10 10 0  6.28318530717959E+0000}%
%
\special{pn 8}%
\special{sh 1}%
\special{ar 4056 1610 10 10 0  6.28318530717959E+0000}%
\special{sh 1}%
\special{ar 4266 1610 10 10 0  6.28318530717959E+0000}%
\special{sh 1}%
\special{ar 4456 1610 10 10 0  6.28318530717959E+0000}%
\special{sh 1}%
\special{ar 4456 1610 10 10 0  6.28318530717959E+0000}%
\put(19.7000,-2.4500){\makebox(0,0){$[1,1]$}}%
\put(29.7000,-2.4000){\makebox(0,0){$[1,2]$}}%
\put(55.7000,-2.5000){\makebox(0,0){$[1,s_1]$}}%
\put(19.7000,-6.5500){\makebox(0,0){$[2,1]$}}%
\put(29.7000,-6.4500){\makebox(0,0){$[2,2]$}}%
\put(55.7000,-6.5500){\makebox(0,0){$[2,s_2]$}}%
\put(19.7000,-17.8500){\makebox(0,0){$[k,1]$}}%
\put(29.7000,-17.8500){\makebox(0,0){$[k,2]$}}%
\put(55.7000,-17.8500){\makebox(0,0){$[k,s_k]$}}%
\put(14.3000,-7.6000){\makebox(0,0){$0$}}%
\special{pn 8}%
\special{sh 1}%
\special{ar 2950 1010 10 10 0  6.28318530717959E+0000}%
\special{sh 1}%
\special{ar 2950 1210 10 10 0  6.28318530717959E+0000}%
\special{sh 1}%
\special{ar 2950 1410 10 10 0  6.28318530717959E+0000}%
\special{sh 1}%
\special{ar 2950 1410 10 10 0  6.28318530717959E+0000}%
\special{pn 8}%
\special{ar 1110 1000 290 220  0.4187469 5.9693013}%
\special{pn 8}%
\special{pa 1368 1102}%
\special{pa 1376 1090}%
\special{fp}%
\special{sh 1}%
\special{pa 1376 1090}%
\special{pa 1324 1138}%
\special{pa 1348 1136}%
\special{pa 1360 1158}%
\special{pa 1376 1090}%
\special{fp}%
\special{pn 8}%
\special{ar 910 1000 510 340  0.2464396 6.0978374}%
\special{pn 8}%
\special{pa 1400 1096}%
\special{pa 1406 1084}%
\special{fp}%
\special{sh 1}%
\special{pa 1406 1084}%
\special{pa 1362 1138}%
\special{pa 1384 1132}%
\special{pa 1398 1152}%
\special{pa 1406 1084}%
\special{fp}%
\special{pn 8}%
\special{sh 1}%
\special{ar 540 1000 10 10 0  6.28318530717959E+0000}%
\special{sh 1}%
\special{ar 620 1000 10 10 0  6.28318530717959E+0000}%
\special{sh 1}%
\special{ar 700 1000 10 10 0  6.28318530717959E+0000}%
\special{pn 8}%
\special{ar 1200 1000 170 100  0.7298997 5.6860086}%
\special{pn 8}%
\special{pa 1314 1076}%
\special{pa 1328 1068}%
\special{fp}%
\special{sh 1}%
\special{pa 1328 1068}%
\special{pa 1260 1084}%
\special{pa 1282 1094}%
\special{pa 1280 1118}%
\special{pa 1328 1068}%
\special{fp}%
\end{picture}%
\end{center}

\vspace{10pt}

Let $\Omega^0$ denote the set of arrows $\gamma\in\Omega$ such that $h(\gamma)\neq t(\gamma)$. 

\begin{lemma} Let $\K$ be any field. Let $\varphihat\in{\rm
    Rep}_{\Gamma,\v}(\K)$ and assume that $v_0> 0$. If $\varphihat$
  is indecomposable, then the linear maps $\varphi_\gamma$, with
  $\gamma\in\Omega^0$, are all injective. 
\label{injective}\end{lemma}

\begin{proof} If $\gamma$ is the arrow $[i,j]\rightarrow [i,j-1]$, with $j=1,\dots,s_i$ and with the convention that $[i,0]=0$, we use the notation $\varphi_{ij}:V_{[i,j]}\rightarrow V_{[i,j-1]}$ rather than $\varphi_\gamma:V_{t(\gamma)}\rightarrow V_{h(\gamma)}$. Assume that $\varphi_{ij}$ is not injective. We define a graded vector subspace $\Vhat'=\bigoplus_{i\in I}V_i'$ of $\Vhat=\bigoplus_{i\in I}V_i$ as follows.

  If the vertex $i$ is not one of the vertices
  $[i,j],[i,j+1],\dots,[i,s_i]$, we put $V'_i:=\{0\}$. We put
  $V'_{[i,j]}:={\rm Ker}\, \varphi_{ij}$,
  $V'_{[i,j+1]}:=\varphi_{i(j+1)}^{-1}(V'_{[i,j]}),\dots,V'_{[i,s_i]}:=\varphi_{is_i}^{-1}(V'_{i(s_i-1)})$. Let
  $\v'$ be the dimension of the graded space $\Vhat'=\bigoplus_{i\in
    I}V_i'$ which we consider as a dimension vector of
  $\Gamma$. Define $\varphihat'\in{\rm Rep}_{\Gamma,\v'}(\K)$ as the
  restriction of $\varphihat$ to $\Vhat'$. It is a non-zero
  subrepresentation of $\varphihat$. It is now possible to define a
  graded vector subspace $\Vhat''=\bigoplus_{i\in I}V_i''$ of $\Vhat$
  such that the restriction $\varphihat''$ of $\varphihat$ to
  $\Vhat''$ satifies $\varphihat=\varphihat''\oplus\varphihat'$: we
  start by taking any subspace $V_{[i,j]}''$ such that
  $V_{[i,j]}=V_{[i,j]}'\oplus V_{[i,j]}''$, then define
  $V_{[i,j+r]}''$ from $V_{[i,j]}''$ as $V_{[i,j+r]}'$ was defined
  from $V_{[i,j]}$, and finally put $V''_i:=V_i$ if the vertex $i$ is
  not one of the vertices $[i,j],[i,j+1],\dots,[i,s_i]$. As $v_0> 0$,
  the subrepresentation $\varphihat''$ is non-zero, and so
  $\varphihat$ is not indecomposable.
\end{proof}

We denote by ${\rm Rep}_{\Gamma,\v}^*(\F_q)$ be the subspace of
representation $\varphihat\in{\rm Rep}_{\Gamma,\v}(\F_q)$ such that
$\varphi_\gamma$ is injective for all $\gamma\in \Omega^0$, and by
${\rm M}_{\Gamma,\v}^*(\F_q)$ the set of isomorphism classes of ${\rm
  Rep}^*_{\Gamma,\v}(\F_q)$. Put $M_{\Gamma,\v}^*(q)=\#\left\{{\rm
    M}_{\Gamma,\v}^*(\F_q)\right\}$.  Following \cite{crawley-par} we
say that a dimension vector $\v$ of $\Gamma$ is \emph{strict} if for
each $i=1,\dots,k$ we have $n_0\geq v_{[i,1]}\geq
v_{[i,2]}\geq\cdots\geq v_{[i,s_i]}$. Let us denote by $\mathcal{S}$
the set of strict dimension vector of $\Gamma$.

\begin{proposition}
$$
\Log\left(\sum_{\v\in\mathcal{S}}M_{\Gamma,\v}^*(q)X^{\v}\right)
=\sum_{\v\in\mathcal{S}-\{0\}}A_{\Gamma,\v}(q)X^{\v}.
$$
\label{hua-inj}
\end{proposition}

\begin{proof} Let us denote by $I_{\Gamma,\v}(q)$ the number of
  isomorphism classes of indecomposable representations in ${\rm
    Rep}_{\Gamma,\v}(\F_q)$. By the Krull-Schmidt theorem, a
  representation of $\Gamma$ decomposes as a direct sum of
  indecomposable representation in a unique way up to permutation of
  the summands. Notice that, for $\v\in\mathcal{S}$, each summand of
  an element of ${\rm Rep}_{\Gamma,\v}^*(\F_q)$ lives in some ${\rm
    Rep}_{\Gamma,\w}^*(\F_q)$ for some $\w\in\mathcal{S}$. On the
  other hand, by Lemma \ref{injective}, ${\rm
    Rep}_{\Gamma,\v}^*(\F_q)$ contains all the indecomposable
  representations in ${\rm Rep}_{\Gamma,\v}(\F_q)$. This implies the
  following identity
$$ 
\sum_{\v\in\mathcal{S}}M_{\Gamma,\v}^*(q)X^{\v}
=\prod_{\v\in\mathcal{S}-\{0\}}(1-X^{\v})^{-I_{\Gamma,\v}(q)},
$$
where $X^{\v}$ denotes the monomial $\prod_{i\in I}X_i^{v_i}$ for some
fixed independent commuting variables $\{X_i\}_{ i\in I}$. 
Exactly as Hua \cite[Proof of Lemma 4.5]{hua} does we show from this
formal identity that  
$$
\Log\left(\sum_{\v\in\mathcal{S}}M_{\Gamma,\v}^*(q)X^{\v}\right)
=\sum_{\v\in\mathcal{S}-\{0\}}A_{\Gamma,\v}(q)\,X^{\v}.
$$ 
\end{proof}
It follows from Proposition~\ref{hua-inj} that since
$A_{\Gamma,\v}(T)\in\Z[T]$ the quantity $M_{\Gamma,\v}^*(q)$ is also
the evaluation of a polynomial with integer coefficients at $T=q$.

Given a non-increasing sequence $u=(n_0\geq n_1\geq \cdots)$ of
non-negative integers we let $\Delta u$ be the sequence of successive
differences $n_0-n_1, n_1-n_2\ldots$. We extend the notation of~\S\ref{charGL}
and denote by $\calF_{\Delta u}$ the set of partial flags of $\F_q$-vector
spaces
$$
\{0\}\subseteq E^r\subseteq\cdots\subseteq E^1\subseteq
E^0=(\F_q)^{n_0}
$$
such that ${\rm dim}(E^i)=n_i$.  

Assume that $\v\in\mathcal S$ and let $\muhat=(\mu^1,\dots,\mu^k)$,
where $\mu^i$ is the partition obtained from $\Delta \v_i$ by
reordering, where $\v_i:=(v_0\geq v_{[i,1]}\geq\cdots\geq
v_{[i,s_i]})$.  Consider the set of orbits
$$
{\mathfrak G}_\muhat(\F_q) :=\left.\left({\rm
      Mat}_{n_0}(\F_q)^g\times\prod_{i=1}^k\calF_{\mu^i}(\F_q)\right)
\right/\GL_{v_0}(\F_q), 
$$
where $\GL_{v_0}(\F_q)$ acts by conjugation on the first $g$
coordinates and in the obvious way on each $\calF_{\mu^i}(\F_q)$.

Let $\varphihat\in {\rm Rep}_{\Gamma,\v}^*(\F_q)$ with underlying
graded vector space $\Vhat=V_0\oplus\bigoplus_{i,j} V_{[i,j]}$. We
choose a basis of $V_0$ and we identify $V_0$ with $(\F_q)^{v_0}$. In
the chosen basis, the $g$ maps $\varphi_\gamma$, with
$\gamma\in\Omega-\Omega^0$, give an element in ${\rm
  Mat}_{v_0}(\F_q)^g$. For each $i=1,\dots,k$, we obtain a partial
flag by taking the images in $(\F_q)^{v_0}$ of the $V_{[i,j]}$'s via
the compositions of the $\varphi_\gamma$'s where $\gamma$ runs over
the arrows of the $i$-th leg of $\Gamma$. We thus have defined a map 
$$
{\rm Rep}_{\Gamma,\v}^*(\F_q)\longrightarrow\left.\left({\rm
      Mat}_{v_0}(\F_q)^g\times
    \prod_{i=1}^k\calF_{\Delta \v_i}(\F_q)\right)\right/\GL_{v_0}(\F_q).
$$
The target set is clearly in bijection with ${\mathfrak G}_\muhat(\F_q)$ and
two elements of ${\rm Rep}_{\Gamma,\v}^*(\F_q)$ have the same image if
and only if they are isomorphic.  Hence this map induces an
isomorphism ${\rm M}_{\Gamma,\v}^*(\F_q)\simeq {\mathfrak
  G}_\muhat(\F_q)$.

We define a new comet-shaped quiver $\Gamma_\muhat$ consisting of $g$
loops on a central vertex and $k$ legs of length $l(\mu^i)-1$ and let
$\v_\muhat$ be the dimension vector as
in~\S\ref{quiver-repns}. Applying the above construction to the pair
$\Gamma_\muhat,\v_\muhat$ we obtain a bijection ${\rm
  M}_{\Gamma_\muhat,\v_\muhat}^*(\F_q)\simeq {\mathfrak
  G}_\muhat(\F_q)$. Put $G_\muhat(q):=\#{\mathfrak G}_\muhat(\F_q)$
and let $A_\muhat(q)$ be the $A$-polynomial of the quiver
$\Gamma_\muhat$ for the dimension vector $\v_\muhat$.
\begin{theorem}
We have 
$$
\Log\left(\sum_{\muhat\in\calP^k}G_\muhat(q)\,m_\muhat\right)
=\sum_{\muhat\in\calP^k-\{0\}}A_\muhat(q)\,m_\muhat. 
$$
\label{MA}
\end{theorem}

\begin{proof} In Proposition \ref{hua-inj} make the change
  of variables 
$$
X_0:=x_{1,1}\cdots x_{k,1}, \qquad
X_{[i,j]}:=x_{i,j}^{-1}x_{i,j+1}, \qquad i=1,2,\ldots,k, \quad
j=1,2,\ldots. 
$$
Since the terms on both sides are invariant under permutation of the
entries $v_{[i,1]},v_{[i,2]},\ldots$ of $\v$  we can collect all terms
that yield the same multipartition $\muhat$. The resulting sum of
$X^\v$ gives the monomial symmetric function $m_\muhat(x)$.
 \end{proof}

\begin{remark} Since $A_\muhat(q)\in\Z[q]$, it follows from Theorem \ref{MA} that $G(q)\in \Z[q]$.
\label{G}
\end{remark}

Recall that $\F$ denotes an algebraic closure of $\F_q$ and $f:\F\rightarrow\F, x\mapsto x^q$ is the Frobenius endomorphism.

\begin{proposition} We have $$\log\left(\sum_\muhat G_\muhat(q)m_\muhat\right)=\sum_{d=1}^\infty \phi_d(q)\cdot\log\left(\Omega\left(\x_1^d,\dots,\x_k^d;0,q^{d/2}\right)\right)$$where $\phi_n(q)=\frac{1}{n}\sum_{d|n}\mu(d)(q^{n/d}-1)$ is the number of $\langle f\rangle$-orbits of $\F^{\times}:=\F-\{0\}$ of size $n$.
\label{sumM}\end{proposition}

\begin{proof}If $X$ is a finite set on which a finite group $H$ acts, recall Burnside's formula which says that

$$\#X/H=\frac{1}{|H|}\sum_{h\in H}\#\{x\in X\,|\, h\cdot x=x\}.$$
 
Denote by ${\bm C}_n$
the set of conjugacy classes of $\GL_n(\F_q)$. Applying Burnside's formula to
${\mathfrak G}_\muhat(\F_q)$, with $\muhat\in(\calP_n)^k$, we find that 

\begin{align*}G_\muhat(q)&=|\GL_n(\F_q)|^{-1}\sum_{g\in \GL_n(\F_q)}\Lambda(g)\prod_{i=1}^k\#\{X\in\calF_{\mu^i}\,|\, g\cdot X=X\}\\
 &=|\GL_n(\F_q)|^{-1}\sum_{g\in \GL_n(\F_q)}\Lambda(g)\prod_{i=1}^kR_{L_{\mu^i}}^G(1)(g)\\
&=\sum_{\calO\in {\bm C}_n}\frac{\Lambda(\calO)}{|Z_\calO|}\prod_{i=1}^kR_{L_{\mu^i}}^G(1)(\calO)
\end{align*}

For a conjugacy class $\calO$ of $\GL_n(\F_q)$, let $\omega(\calO)$ denotes its type. By Formula (\ref{alambda}), we have 

$$\frac{\Lambda(\calO)}{|Z_\calO|}=\calH_{\omega(\calO)}(0,\sqrt{q}).$$By Corollary \ref{R}, we deduce that

$$\sum_\muhat G_\muhat(q)m_\muhat=\sum_{\calO\in {\bm C}}\calH_{\omega(\calO)}(0,\sqrt{q})\prod_{i=1}^k\tilde{H}_{\omega(\calO)}(\x_i,q)$$where ${\bm C}:=\bigcup_{n\geq 1}{\bm C}_n$.

We denote by ${\bf F}^{\times}$ the set of $\langle f\rangle$-orbits of $\F^{\times}$. There is a natural bijection from the set ${\bm C}$ to the set of all maps ${\bf F}^{\times}\rightarrow\calP$ with finite support \cite[IV, 2]{macdonald}. If $C\in{\bm C}$ corresponds to $\alpha:{\bf F}^{\times}\rightarrow\calP$, then we may enumerate the elements of $\{s\in{\bf F}^{\times}\,|\, \alpha(s)\neq 0\}$ as $c_1,\dots,c_r$ such that $\omega(\alpha):=(d(c_1),\alpha(c_1))\cdots(d(c_r),\alpha(c_r))$, where $d(c)$ denotes the size of $c$, is the type $\omega(C)$. 

We have 

\begin{align*}\sum_\muhat G_\muhat(q)m_\muhat&=\sum_{\alpha\in \calP^{{\bf F}^{\times}}}\calH_{\omega(\alpha)}(0,\sqrt{q})\prod_{i=1}^k\tilde{H}_{\omega(\alpha)}(\x_i,q)\\
&=\prod_{c\in{\bf F}^{\times}}\Omega\left(\x_1^{d(c)},\dots,\x_k^{d(c)};0,q^{d(c)/2}\right)\\
&=\prod_{d=1}^\infty\Omega\left(\x_1^d,\dots,\x_k^d;0,q^{d/2}\right)^{\phi_d(q)}
\end{align*}
\end{proof}

\begin{remark}The second formula displayed in the proof of Proposition \ref{sumM} shows that

$$
G_\muhat(q)=\left\langle\Lambda\otimes
  R_\muhat(1),1\right\rangle$$where
$R_\muhat(1):=R_{L_{\mu^1}}^G(1)\otimes\cdots\otimes
R_{L_{\mu^k}}^G(1)$. 
\label{rem326}\end{remark}

\begin{theorem}
We have $$A_\muhat(q)=\H_\muhat(0,\sqrt{q}).$$
 \label{purity}
\end{theorem}

\begin{proof}  From Formula (\ref{exp}) we have

$$
\sum_\muhat\H_\muhat(0,\sqrt{q})\,m_\muhat=(q-1)\, \Log\,\left(\Omega(0,\sqrt{q})\right).
$$
We thus need to see that
\beq
\sum_\muhat A_\muhat(q)\,m_\muhat=(q-1)\, \Log\,\left(\Omega(0,\sqrt{q})\right).
\label{ExpA}
\eeq

From Theorem \ref{MA} we are reduced to prove that

$$
\Log\left(\sum_\muhat G_\muhat(q)m_\muhat\right)=(q-1)\Log\left(\Omega(0,\sqrt{q})\right).
$$ 

But this follows from Lemma \ref{moz} and Proposition \ref{sumM}.

\end{proof}

\subsection{Another formula for the $A$-polynomial}

When the dimension vector $\v_\muhat$ is indivisible, it is known by Crawley-Boevey and van den Bergh \cite{crawley-boevey-etal} that the polynomial $A_\muhat(q)$ equals (up to some power of $q$) to the polynomial which counts the number of points of some quiver variety over $\F_q$.

Here we prove some relation between $A_\muhat(q)$ and some variety which is closely related to quiver varieties. This relation holds for any $\muhat$ (in particular $\v_\muhat$ can be divisible).

We continue to use the notation $G$, $P_\lambda$, $L_\lambda$, $U_\lambda$, $\calF_\lambda$ of \S \ref{finite-groups}  and the notation $\mathfrak{g}$, $\mathfrak{p}_\lambda$, $\mathfrak{l}_\lambda$, $\mathfrak{u}_\lambda$ of \S \ref{Fourier}.

For a partition $\lambda$ of $n$, define $$\X_\lambda:=\left\{(X,gP_\lambda)\in \mathfrak{g}\times(G/P_\lambda)\,\left|\, g^{-1}Xg\in\mathfrak{u}_\lambda\right\}\right.$$It is well-known that the image of the projection $p:\X_\lambda(\F)\rightarrow \mathfrak{g}(\F)$, $(X,gP_\lambda)\mapsto X$ is the Zariski closure $\overline{\calO}_{\lambda'}$ of the nilpotent adjoint orbit $\calO_{\lambda'}$ of $\gl_n(\F)$ whose Jordan form is given by $\lambda'$, and that $p$ is a desingularization.

Put

$$\mathbb{V}_\muhat:=\left\{\left(a_1,b_1,\dots,a_g,b_g,(X_1,g_1P_{\mu^1}),\dots, (X_k,g_kP_{\mu^k})\right)\in\mathfrak{g}^{2g}\times\X_{\mu^1}\times\cdots\times\X_{\mu^k}\,\left|\, \sum_i[a_i,b_i]+\sum_jX_j=0\right\}\right.$$where $[a,b]=ab-ba$. 

Define $\Lambda^\sim:\mathfrak{g}\rightarrow\C$, $z\mapsto q^{gn^2}\Lambda(z)$. By \cite[Proposition 3.2.2]{hausel-letellier-villegas} we know that 

$$\Lambda^\sim=\calF^\mathfrak{g}(F)$$where for $z\in\mathfrak{g}$, $$F(z):=\#\left\{(a_1,b_1,\dots,a_g,b_g)\in\mathfrak{g}^{2g}\,\left|\, \sum_i[a_i,b_i]=z\right.\right\}.$$

By Remark \ref{charab}, the functions $\Lambda^\sim$ and $\mathfrak{R}_{\mathfrak{l}_\lambda}^\mathfrak{g}:=q^{\frac{1}{2}(n^2-\sum_i\lambda_i^2)}R_{\mathfrak{l}_\lambda}^\mathfrak{g}$ are characters of $\mathfrak{g}$. Put $$\mathfrak{R}_\muhat(1):=\mathfrak{R}_{\mathfrak{l}_{\mu^1}}^\mathfrak{g}(1)\otimes\cdots\otimes \mathfrak{R}_{\mathfrak{l}_{\mu^k}}^\mathfrak{g}(1).$$

For two functions $f,g:\mathfrak{g}\rightarrow\C$, define their inner product as

$$
\langle f,g\rangle=|\mathfrak{g}|^{-1}\sum_{X\in\mathfrak{g}}f(X)\overline{g(X)}.
$$

\begin{proposition}We have $$|\mathbb{V}_\muhat|=\left\langle \Lambda^\sim\otimes \mathfrak{R}_\muhat(1),1\right\rangle.$$
\label{inner}\end{proposition}

\begin{proof}Notice that

$$
|\mathbb{V}_\muhat|=\left(F*Q_{\mathfrak{l}_{\mu^1}}^\mathfrak{g}*\cdots*Q_{\mathfrak{l}_{\mu^k}}^\mathfrak{g}\right)(0).
$$
Hence the result follows from Proposition \ref{fourprop1} and Proposition \ref{fourprop2}.
\end{proof}

The proposition shows that $|\mathbb{V}_\muhat|$ is a rational function in $q$ which is an integer for infinitely many values of $q$. Hence  $|\mathbb{V}_\muhat|$ is a polynomial in $q$ with integer coefficients.

Consider 

$$
V_\muhat(q):=\frac{|\mathbb{V}_\muhat|}{|G|}.
$$

Recall that $d_\muhat=n^2(2g-2+k)-\sum_{i,j}(\mu^i_j)^2+2$.

\begin{theorem}
We have 

$$\Log\left(\sum_\muhat q^{-\frac{1}{2}(d_\muhat-2)}V_\muhat(q)m_\muhat\right)=\frac{q}{q-1}\sum_\muhat A_\muhat(q)m_\muhat.$$
\label{theohua}
\end{theorem}

By Lemma \ref{moz} and Formula (\ref{ExpA}) we are reduced to prove the following.

\begin{proposition} We have $$\log\,\left(\sum_\muhat  q^{-\frac{1}{2}(d_\muhat-2)}V_\muhat(q)m_\muhat\right)=\sum_{d=1}^\infty\varphi_d(q)\cdot\log\,\left(\Omega\left(\x_1^d,\dots,\x_k^d;0,q^{d/2}\right)\right)$$where $\varphi_n(q)=\frac{1}{n}\sum_{d|n}\mu(d)q^{n/d}$ is the number of $\langle f\rangle$-orbits of $\F$ of size $n$.
\end{proposition}

\begin{proof} By Proposition \ref{inner}, we have

$$V_\muhat(q)=\frac{q^{-n^2+\frac{1}{2}(kn^2-\sum_{i,j}(\mu^i_j)^2)}}{|G|}\sum_{x\in\mathfrak{g}}\Lambda^\sim(x)R_{\mathfrak{l}_{\mu^1}}^\mathfrak{g}(1)(x)\cdots R_{\mathfrak{l}_{\mu^1}}^\mathfrak{g}(1)(x).$$By Remark \ref{Rl=RL} and Corollary \ref{R}, we see that $R_{\mathfrak{l}_\lambda}^\mathfrak{g}(1)(x)=\left\langle\tilde{H}_\omega(\x;q),h_\lambda(\x)\right\rangle$ when the $G$-orbit of $x$ is of type $\omega$.

We now proceed exactly as in the proof of Proposition \ref{sumM} to prove our formula.
\end{proof}

\subsection{Applications to the character theory of finite general
  linear groups}\label{applichar} 

The following theorem (which is a consequence of Theorem \ref{purity}
and Theorem \ref{multi}) expresses certain fusion rules in the
character ring of $\GL_n(\F_q)$ in terms of absolutely indecomposable
representations of comet shaped quivers.

\begin{theorem}
\label{mult-thm}
We have 
 $$
\langle \Lambda\otimes R_\muhat,1\rangle=A_\muhat(q).
$$
\label{multi=A}\end{theorem}

From Theorem \ref{multi=A} and Theorem \ref{kactheo} we have the following result.

\begin{corollary} $\langle \Lambda\otimes R_\muhat,1\rangle\neq 0$ if
  and only if $\v_\muhat\in\Phi(\Gamma_\muhat)^+$. Moreover $\langle
  \Lambda\otimes R_\muhat,1\rangle=1$ if and only if $\v_\muhat$ is a
  real root.
\end{corollary}

\begin{remark}
  We will see in \S \ref{delta-non-neg} that $\v_\muhat$ is always an
  imaginary root when $g\geq 1$, hence the second assertion concerns
  only the case $g=0$ (i.e. $\Lambda=1$).
\end{remark}

A proof of Theorem~\ref{mult-thm} for $\v_\muhat$ is indivisible is
given in~\cite{hausel-letellier-villegas} by expressing $\langle
\Lambda\otimes R_\muhat,1\rangle$ as the Poincar\'e polynomial of a
comet-shaped quiver variety. This quiver variety exists only when
$\v_\muhat$ is indivisible.

In \cite{Hiss} the authors discusses some results of Mattig who showed
that the multiplicities $\langle\calX_1\otimes\calX_2,\calX_3\rangle$,
with $\calX_1,\calX_2,\calX_3$ unipotent characters of $\GL_n$, are
polynomials in $q$ with rational coefficients. Using calcultion with
CHEVIE, he also observed when $n\leq 8$ that the coefficients of
$\langle\calX_1\otimes\calX_2,\calX_3\rangle$ are non-negative
integers. The following proposition confirms part of this prediction.

\begin{proposition}
 For a $k$-tuple of partitions
  $\muhat=(\mu^1,\ldots,\mu^k)$ of $n$ there exists a polynomial
  $U_\muhat\in \Z[T]$ such that $\langle
  \Lambda\otimes \calU_{\mu^1}\otimes\cdots\otimes
  \calU_{\mu^k},1\rangle=U_\muhat(q)$. 
\end{proposition}
\begin{proof}
  Since the complete symmetric functions $\{h_\lambda(\x)\}_\lambda$
  forms a $\Z$-basis of the ring $\Lambda(\x)$, for any partitions
  $\mu,\lambda$, there exist integers $a_{\mu\lambda}$ such
  that $$s_\mu(\x)=\sum_\lambda a_{\mu\lambda}h_\lambda(\x).$$The
  matrix $(a_{\mu\lambda})_{\mu,\lambda}$ is the transpose inverse of
  the matrix $K=(K_{\mu\lambda})_{\mu,\lambda}$ of Kostka numbers.

  If $\calU_\mu$ is the unipotent characters of $\GL_n$ corresponding
  to $\mu$, by Theorem \ref{Rtau} we have $$\calU_\mu(C)=\left\langle
    \tilde{H}_\omega(\x;q),s_{\mu}(\x)\right\rangle$$where $C$ is a
  conjugacy class of type $\omega$. Hence by Corollary \ref{R}, we
  deduce that $$\calU_\mu=\sum_\lambda
  a_{\mu\lambda}R_{L_\lambda}^G(1).$$Hence
  $\left\langle\Lambda\otimes\calU_1\otimes\cdots\otimes\calU_k,1\right\rangle$
  is a $\Z$-linear combination of multiplicities of the
  form $$\left\langle\Lambda\otimes
    R_{L_{\lambda^1}}^G(1)\otimes\cdots\otimes
    R_{L_{\lambda^k}}^G(1),1\right\rangle$$and so by Remark
  \ref{rem326}, it is a $\Z$-linear combination of polynomials of the
  form $G_\muhat(q)$ which have integer coefficients (see Remark \ref{G}).
  \end{proof}

\section{Example: Hilbert Scheme of $n$ points on $\C^\times\times\C^\times$}\label{Hilbert}

Throughout this section we will have $g=k=1$ and $\muhat$ will be
either the partition $(n)$ or $(n-1,1)$.  

In this section we illustrate our conjectures and formulas in these cases.

\subsection{Hilbert schemes: Review}

For a nonsingular complex surface $S$ we denote by $S^{[n]}$ the
Hilbert scheme of $n$ points in $S$. Recall that $S^{[n]}$ is
nonsingular and has dimension $2n$. 

We denote by $Y^{[n]}$ the Hilbert scheme of $n$ points in $\C^2$. 

Recall (see for instance \cite[\S 5.2]{NakHilbert}) that
$h_c^i(Y^{[n]})=0$  unless $i$ is even and that the compactly
supported Poincar\'e polynomial
$P_c(Y^{[n]};q):=\sum_ih_c^{2i}(Y^{[n]})q^i$ is given by the following
explicit formula  

\begin{equation}\sum_{n\geq 0}P_c(Y^{[n]};q)T^n=\prod_{m\geq 1}\frac{1}{1-q^{m+1}T^m}.\label{Yn}
\end{equation}which is equivalent to  

\begin{equation}\Log\left(\sum_{n\geq 0}q^{-n}\cdot P_c(Y^{[n]};q)T^n\right)=\sum_{n\geq 1}q T^n.\label{Ynbis}\end{equation}

For $n\geq 2$, consider the partition $\mu=(n-1,1)$ of $n$ and let $C$
be a semisimple adjoint orbit of $\gl_n(\C)$ with characteristic
polynomial of the form $(-1)^n(x-\alpha)^{n-1}(x-\beta)$ with
$\beta=-(n-1)\alpha$ and $\alpha\neq 0$. Consider the
variety 

$$
\calV_{(n-1,1)}=\{(a,b,X)\in (\gl_n)^2\times C\,|\,
[a,b]+X=0\}.
$$

The group $\GL_n$ acts on $\calV_{(n-1,1)}$ diagonally
by conjugating the coordinates. This action induces a free action of
$\PGL_n$ on $\calV_{(n-1,1)}$ and we
put 
$$
\calQ_{(n-1,1)}:=\V_{(n-1,1)}/\!/\PGL_n={\rm
  Spec}\left(\C[\V_{(n-1,1)}]^{\PGL_n}\right).
  $$
  The variety
$\calQ_{(n-1,1)}$ is known to be nonsingular of dimension $2n$ (see
for instance \cite[\S 2.2]{hausel-letellier-villegas} and the
references therein).

We have the following well-known theorem.

\begin{theorem}
 The two varieties $\calQ_{(n-1,1)}$ and $Y^{[n]}$ have isomorphic
 cohomology supporting pure mixed Hodge structures. 
\label{adpure}
\end{theorem}

\begin{proof} 
  By \cite[Appendix B]{hausel-letellier-villegas} it is enough to
  prove that there is a smooth morphism $f:\mathfrak{M}\rightarrow\C$
  which satisfies the two following properties:

\noindent (1) There exists an action of $\C^\times$ on $\mathfrak{M}$ such that the fixed point set $\mathfrak{M}^{\C^\times}$ is complete and for all $x\in X$ the limit ${\rm lim}_{\lambda\mapsto 0}\lambda x$ exists.

\noindent (2)  $\calQ_{(n-1,1)}=f^{-1}(\lambda)$ and $Y^{[n]}=f^{-1}(0)$. 

Denote by $\v$ the dimension vector of $\Gamma_{(n-1,1)}$ which has coordinate $n$ on the central vertex (i.e., the vertex supporting the loop) and $1$ on the other vertex. It is well-known (see Nakajima \cite{NakHilbert}) that $Y^{[n]}$ can be identified with the quiver variety  $\mathfrak{M}_{0,\theta}(\v)$ where $\theta$ is the stability parameter with coordinate $-1$ on the central vertex and $n$ on the other vertex. If we let $\xi$ be the parameter with coordinate $-\alpha$ at the central vertex and $\alpha-\beta$ at the other vertex, then the variety $\calQ_{(n-1,1)}$ is isomorphic to the quiver variety $\mathfrak{M}_{\xi,\theta}(\v)$ (see for instance \cite{hausel-letellier-villegas} and the references therein). Now we can define as in \cite[\S 2.2]{hausel-letellier-villegas} a map $f:\mathfrak{M}\rightarrow\C$ such that $f^{-1}(0)=\mathfrak{M}_{0,\theta}(\v)$ and $f^{-1}(\lambda)=\mathfrak{M}_{\xi,\theta}(\v)$ and which satisfies the required properties.
\end{proof}

\begin{proposition} We have $$P_c(Y^{[n]};q)=q^n\cdot A_{(n-1,1)}(q).$$
\label{Ypure}\end{proposition}

\begin{proof} We have $P_c(\calQ_{(n-1,1)};q)=q^n\cdot \H_{(n-1,1)}(0,\sqrt{q})$ by \cite[Theorem 1.3.1]{hausel-letellier-villegas} and so by Theorem \ref{purity} we see that $P_c(\calQ_{(n-1,1)};q)=q^n\cdot A_{(n-1,1)}(q)$. Hence the result follows from Theorem \ref{adpure}.\end{proof}

Now put $X:=\C^*\times\C^*$. Unlike $Y^{[n]}$, the mixed Hodge structure on $X^{[n]}$ is not pure. By G\"ottsche and Soergel \cite{Gottsche-Soergel} we have the following result.

\begin{theorem}We have $h_c^{i,j;k}(X^{[n]})=0$ unless $i=j$ and 

\begin{equation}
1+\sum_{n\geq 1}H_c\big(X^{[n]};q,t)T^n=\prod_{n\geq
  1}\frac{(1+t^{2n+1}q^nT^n)^2}{(1-q^{n-1}t^{2n}T^n)(1-t^{2n+2}q^{n+1}
  T^n)}
\label{GS}\end{equation}with $H_c\left(X^{[n]};q,t\right):=\sum_{i,k}h_c^{i,i;k}(X^{[n]})q^it^k$.
\end{theorem}

Define $\H^{[n]}(z,w)$ such
that 
$$
H_c\left(X^{[n]};q,t\right)
=(t\sqrt{q})^{2n}\H^{[n]}\left(-t\sqrt{q},\frac{1}{\sqrt{q}}\right).
$$  
Then Formula (\ref{GS}) reads
\begin{equation}
  \sum_{n\geq 0}\H^{[n]}(z,w)T^n=\prod_{n\geq
    1}\frac{(1-zwT^n)^2}{(1-z^2T^n)(1-w^2T^n)},
\label{GS1}
\end{equation}
with the convention that $\H^{[0]}(z,w)=1$. Hence we may
re-write Formula (\ref{GS}) as
\begin{equation}
\Log\left(\sum_{n\geq
    0}\H^{[n]}(z,w)T^n\right)=\sum_{n\geq 1}(z-w)^2T^n. 
\label{GS2}\end{equation}

Specializing Formula (\ref{GS2}) with $(z,w)\mapsto (0,\sqrt{q})$ we
see from Formula (\ref{Ynbis})
that \begin{equation}P_c(Y^{[n]};q)=q^n\cdot
  \H^{[n]}(0,\sqrt{q}).\label{PH=P}\end{equation}We thus have the
following result.

\begin{proposition} We have 

$$
PH_c(X^{[n]};T)=P_c(Y^{[n]};T).
$$
where $PH_c(X^{[n]};T):=\sum_ih_c^{i,i;2i}(X^{[n]})T^i$ is the Poincar\'e polynomial of the pure part of the cohomology of $X^{[n]}$.

\end{proposition}

\subsection{A conjecture}

The aim of this section is to discuss the following conjecture.

\begin{conjecture}
We have 
\begin{equation}\H_{(n-1,1)}(z,w)=\H^{[n]}(z,w).
\label{CV=HS}
\end{equation}
\label{conjCV=HS}
\end{conjecture}

Modulo the conjectural formula (\ref{mainconj}), Formula (\ref{CV=HS}) says that the
two mixed Hodge polynomials $H_c(X^{[n]};q,t)$ and
$H_c(\M_{(n-1,1)};q,t)$ agree. This would be a multiplicative analogue
of Theorem \ref{adpure}. Unfortunately the proof of Theorem
\ref{adpure} does not work in the multiplicative case. This is because
the natural family $g:\mathfrak{X}\to \C$ with $X^{[n]}=g^{-1}(0)$ and
$\M_{(n-1,1)}=g^{-1}(\lambda)$ for $0\neq \lambda \in \C$ does not
support a $\C^\times$-action with a projective fixed point set and so
\cite[Appendix B]{hausel-letellier-villegas} does not apply.

One can still attempt to prove that the restriction map
$H^*(\mathfrak{X};\Q)\to H^*(g^{-1}(\lambda);\Q)$ is an isomorphism
for every fibre over $\lambda\in \C$ by using a family version of the
non-Abelian Hodge theory as developed in the tamely ramified case in
\cite{simpson}. In other words one would construct a family
$g_\Dol:\mathfrak{X}_{\Dol} \to \C$ such that $g^{-1}_\Dol(0)$ would
be isomorphic with the moduli space of parabolic Higgs bundles on an
elliptic curve $C$ with one puncture and flag type $(n-1,1)$ and
meromorphic Higgs field with a nilpotent residue at the puncture, and
$g^{-1}_\Dol(\lambda)$ for $\lambda\neq 0$ would be isomorphic with
parabolic Higgs bundles on $C$ with one puncture and semisimple
residue at the puncture of type $(n-1,1)$. In this family one should
have a $\C^\times$ action satisfying the assumptions of \cite[Appendix
B]{hausel-letellier-villegas} and so could conclude that
$H^*(\mathfrak{X}_\Dol;\Q)\to H^*(g_\Dol^{-1}(\lambda);\Q)$ is an
isomorphism for every fibre over $\lambda\in \C$. Then a family
version of non-Abelian Hodge theory in the tamely ramified case would
yield that the two families $\mathfrak{X}_\Dol$ and $\mathfrak{X}$ are
diffeomorphic, and so one could conclude the desired isomorphism
$H^*(X^{[n]};\Q)\cong H^*(\M_{(n-1,1)})$ preserving mixed Hodge
structures. However a family version of the non-Abelian Hodge theory
in the tamely ramified case (which was initiated in \cite{simpson}) is
not available in the literature.

\begin{proposition}
Conjecture~\ref{conjCV=HS} is true under the specialization
$z=0,w=\sqrt q$.
\end{proposition}
\begin{proof} The left hand side specializes to $A_{(n-1,1)}(q)$ by
  Theorem~\ref{purity}, which by~\eqref{GS2} and
  Proposition~\ref{Ypure} agrees with the right hand side.
\end{proof}

The {\it Young diagram} of a partition
$\lambda=(\lambda_1,\lambda_2,\dots)$ is defined as the set of points
$(i,j)\in\Z^2$ such that $1\leq j\leq\lambda_i$. We adopt the
convention that the coordinate $i$ of $(i,j)$ increases as one goes
down and the second coordinate $j$ increases as one goes to the
right.

For $\lambda\neq 0$, we define
$\phi_\lambda(z,w):=\sum_{(i,j)\in\lambda}z^{j-1}w^{i-1}$, and for
$\lambda=0$, we put $\phi_\lambda(z,w)=0$.  Define
\begin{align*}
&A_1(z,w;T):=\sum_\lambda\calH_\lambda(z,w)\phi_\lambda(z^2,w^2)T^{|\lambda|},\\ 
        &A_0(z,w;T):=\sum_\lambda\calH_\lambda(z,w)T^{|\lambda|}.
\end{align*}

\begin{proposition}
\label{Log-fmla}
We have
$$
\sum_{n\geq
  1}\H_{(n-1,1)}(z,w)T^n=(z^2-1)(1-w^2)\frac{A_1(z,w;T)}{A_0(z,w;T)}.
$$
\end{proposition}

\begin{proof}
  The coefficient of the monomial symmetric function $m_{(n-1,1)}(\x)$
  in a symmetric function in $\Lambda(\x)$ of homogeneous degree $n$
  is the coefficient of $u$ when specializing the variables
  $\x=\{x_1,x_2,\dots\}$ to $\{1,u,0,0\dots\}$.  Hence, the generating
  series $\sum_{n\geq 1}\H_{(n-1,1)}(z,w)T^n$ is the coefficient of
  $u$ in
$$
(z^2-1)(1-w^2)\Log\left(\sum_\lambda\calH_\lambda(z,w)
  \tilde{H}_\lambda(1,u,0,0,\dots;z^2,w^2)\,T^{|\lambda|}\right).
$$ 
We know
that 
$$
\tilde{H}_\lambda(\x;z,w)
=\sum_\rho\tilde{K}_{\rho\lambda}(z,w)s_\rho(\x),
$$
and $s_\rho(\x)=\sum_{\mu\unlhd\rho} K_{\rho\mu}m_\mu(\x)$ where
$K_{\rho\mu}$ are the Kostka numbers. We have
\begin{align*}
&s_{(n)}(1,u,0,0,\dots)=1+u+O(u^2)\\
&s_{(n-1,1)}(1,u,0,0,\dots)=u+O(u^2)
\end{align*}
and
$$
s_\rho(1,u,0,0,\dots)=O(u^2)
$$
for any other partition $\rho$. Hence,
$$
\tilde{H}_\lambda(1,u,0,0,\dots;z,w)
=\tilde{K}_{(n)\lambda}(z,w)(1+u)
+\tilde{K}_{(n-1,1)\lambda}(z,w)u+O(u^2). 
$$
From Macdonald \cite[p. 362]{macdonald} we obtain
$\tilde{K}_{(n)\lambda}(a,b)=1$ and
$\tilde{K}_{(n-1,1)\lambda}(a,b)=\phi_\lambda(a,b)-1$.
Hence, finally,
\begin{equation}
\tilde{H}_\lambda(1,u,0,0,\dots;z,w)=
1+\phi_\lambda(z,w)u+O(u^2).
\end{equation}

It follows that $(z^2-1)^{-1}(1-w^2)^{-1}\sum_{n\geq
  1}\H_{(n-1,1)}(z,w)T^n$ equals the coefficient of $u$ in
$$
\Log\left(\sum_\lambda\calH_\lambda(z,w)\left(1
    +\phi_\lambda(z^2,w^2)u+O(u^2)\right) 
  T^{|\lambda|}\right)=\Log\left(A_0(T)+A_1(T)u+O(u^2)\right).
$$
The claim follows from the general fact
$$
\Log\left(A_0(T)+A_1(T)u+O(u^2)\right)=\Log \,
A_0(T)+\frac{A_1(T)}{A_0(T)}u+O(u^2).
$$
 \end{proof}

 Combining Proposition~\ref{Log-fmla} with~\eqref{GS1} we obtain
 the following.
\begin{corollary}
  Conjecture~\ref{conjCV=HS} is equivalent to the following
  combinatorial identity
\begin{equation}
1+(z^2-1)(1-w^2)\frac{A_1(z,w;T)}{A_0(z,w;T)}
 =\prod_{n\geq 1}\frac{(1-zwT^n)^2}{(1-z^2T^n)(1-w^2T^n)}.
\label{comb}
\end{equation}
\end{corollary}

The main result of this section is the following theorem.

\begin{theorem} 
\label{euler-spec}
Formula (\ref{comb}) is true
  under the Euler specialization
  $(z,w)\mapsto\left(\sqrt{q},1/\sqrt{q}\right)$; namely, we have
\begin{equation}
\H_{(n-1,1)}(z,z^{-1})
=\H^{[n]}(z,z^{-1}).
\label{CV=HS1}
\end{equation}
Equivalently, the two varieties $\M_{(n-1,1)}$ and $X^{[n]}$ have the
same $E$-polynomial.
\end{theorem}
\begin{proof}
  Consider the generating  function 
$$
F:=(1-z)(1-w)\sum_\lambda\phi_\lambda(z,w)T^{|\lambda|}.
$$

It is straightforward to see that for $\lambda\neq 0$ we have
\begin{align*}
(1-z)(1-w)\phi_\lambda(z,w)&=1+\sum_{i=1}^{l(\lambda)}(w^i-w^{i-1})z^{\lambda_i}
-w^{l(\lambda)}\\
&=1+\sum_{i\geq 1}(w^i-w^{i-1})z^{\lambda_i}.
\end{align*}
Interchanging summations we find
$$
F=\sum_{i\geq 1}(w^i-w^{i-1})\sum_{\lambda\neq
  0}z^{\lambda_i}T^{|\lambda|}+\sum_{\lambda\neq 0}T^{|\lambda|}.
$$
To compute the sum over $\lambda$ for a fixed $i$ we break the
partitions as follows: 
$$
\lambda_1\geq\lambda_2\geq\cdots\geq
\lambda_{i-1}\geq\underbrace{\lambda_i\geq\lambda_{i+1}\geq\cdots}_{\rho}
$$
and we put
\begin{align*}
&\rho:=(\lambda_i,\lambda_{i+1},\dots)\\
&\mu:=(\lambda_1-\lambda_i,\lambda_2-\lambda_i,\dots,\lambda_{i-1}-\lambda_i)
\end{align*}
Notice that $\mu_1'=l(\mu)<i$, $\rho_1=l(\rho')=\lambda_i$ and
$|\lambda|=|\mu|+|\rho|+l(\rho')(i-1)$.

We then have $$\sum_\lambda
z^{\lambda_i}T^{|\lambda|}=\sum_{\mu_1<i}T^{|\mu|}\sum_\rho
z^{l(\rho)}T^{|\rho|+(i-1)l(\rho)}$$(changing $\rho$ to $\rho'$ and
$\mu$ to $\mu'$). Each sum can be written as an infinite product,
namely
$$
\sum_\lambda
z^{\lambda_i}T^{|\lambda|}=\prod_{k=1}^{i-1}(1-T^k)^{-1}\prod_{n\geq
  1}(1-zT^{n+i-1})^{-1}.
$$

So 
\begin{align*}
F&=\sum_{\lambda\neq 0}T^{|\lambda|}+\sum_{i\geq
    1}(w^i-w^{i-1})\left(\prod_{k=1}^{i-1}(1-T^k)^{-1}\prod_{n\geq
      1}(1-zT^{n+i-1})^{-1}-1\right)\\&=\sum_{\lambda\neq
    0}T^{|\lambda|}+\prod_{n\geq 1}(1-zT^n)^{-1}\sum_{i\geq
    1}(w^i-w^{i-1})\prod_{k=1}^{i-1}\frac{(1-zT^k)}{(1-T^k)}-\sum_{i\geq
    1}(w^i-w^{i-1}).
\end{align*}
The last sum telescopes to $1$ and we find
\begin{equation}
F=\sum_\lambda T^{|\lambda|}+\prod_{n\geq
  1}(1-zT^n)^{-1}(w-1)\sum_{i\geq 1}w^{i-1}\prod_{k=1}^{i-1} 
\frac{(1-zT^k)}{(1-T^k)}.
\label{F}
\end{equation}

By the Cauchy $q$-binomial theorem the sum
equals $$\frac{1}{(1-w)}\prod_{n\geq
  1}\frac{(1-wzT^n)}{(1-wT^n)}.$$Also$$\sum_\lambda
T^{|\lambda|}=\prod_{n\geq 1}(1-T^n)^{-1}.$$If we divide Formula (\ref{F}) by this we
finally get

$$
1-(1-z)(1-w)\prod_{n\geq
    1}(1-T^n)\sum_\lambda\phi_\lambda(z,w)T^{|\lambda|}=\prod_{n\geq
    1}\frac{(1-wzT^n)(1-T^n)}{(1-zT^n)(1-wT^n)}.
 $$

Putting now $(z,w)=(q,1/q)$ we find that

\begin{equation}
1-(1-q)(1-1/q)\prod_{n\geq
  1}(1-T^n)\sum_\lambda\phi_\lambda(q,1/q)T^{|\lambda|}=\prod_{n\geq
  1}\frac{(1-T^n)^2}{(1-qT^n)(1-q^{-1}T^n)}.\label{phi}
  \end{equation}

From Formula (\ref{H-specializ}) we have
$\calH_\lambda(\sqrt{q},1/\sqrt{q})=1$ and
so 
\begin{align*}
&A_1\left(\sqrt{q},\frac{1}{\sqrt{q}};T\right)
=\sum_\lambda\phi_\lambda\left(q,\frac{1}{q}\right)T^{|\lambda|}\\
&A_0\left(\sqrt{q},\frac{1}{\sqrt{q}};T\right)=\sum_\lambda 
  T^{|\lambda|}=\prod_{n\geq 1}(1-T^n)^{-1}
\end{align*}
Hence, under the specialization $(z,w)\mapsto(\sqrt{q},1/\sqrt{q})$ , the left hand side of Formula (\ref{comb}) agrees with the left hand side of Formula (\ref{phi}). 

Finally, it is straightforward to see that if we put $(z,w)=(\sqrt{q},1/\sqrt{q})$,  then the right hand side of Formula   (\ref{comb}) agrees with the right hand side of Formula (\ref{phi}), hence the theorem.

\end{proof}

\subsection{Connection with modular forms}

For a positive, even integer $k$ let $G_k$ be the standard Eisenstein
series for $SL_2(\Z)$
\begin{equation}
\label{Eisenstein-defn}
G_k(T)=\frac{-B_k}{2k}+\sum_{n\geq 1}\sum_{d\,|\, n} d^{k-1}T^n,
\end{equation}
where $B_k$ is the $k$-th Bernoulli number. 

For $k>2$ the $G_k$'s are
modular forms of weight $k$; i.e., they are holomorphic (including at
infinity) and satisfy
\begin{equation}\label{quasi-mod-2}
\begin{split}
&G_k \left( \frac{a\tau +b}{c\tau +d}\right) 
= (c\tau +d)^k G_k (\tau)\\
\noalign{\vskip6pt}
& \text{for }\ \begin{pmatrix}a&b\\ c&d\end{pmatrix} \in SL_2(\Z)\ ,
\quad T = e^{2\pi i\tau}\ ,\quad \Im \tau >0\ .
\end{split}
\end{equation}
For $k=2$ we have a similar transformation up to an additive term.
\begin{equation}\label{transform}
  G_2 \left( \frac{a\tau +b}{c\tau +d}\right) 
  = (c\tau +d)^2 G_2 (\tau) - \frac{c}{4\pi i} (c\tau +d).
\end{equation}
The ring $\Q [G_2,\ G_4,\ G_6]$ is called the ring of
\emph{quasi-modular} forms (see~\cite{kaneko-zagier}).

\begin{theorem}
We have 
$$
1+\sum_{n\geq
  1}\H_{(n-1,1)}\left(e^{u/2},e^{-u/2}\right)T^n
=\frac{1}{u}\left(e^{u/2}-e^{-u/2}\right)\exp\left(2\sum_{k\geq
    2}G_k(T)\frac{u^k}{k!}\right). 
$$
In particular, the coefficient of any power of $u$ on the left hand
side is in the ring of quasi-modular forms.
\end{theorem}
\begin{remark}
  The relation between the $E$-polynomial of the Hilbert scheme of
  points on a surface and theta functions goes back to G\"ottsche
  \cite{Gottsche}.
\end{remark}

\begin{proof}
Consider the classical theta function 
\begin{equation}\label{classical-theta}
\theta (w) = (1-w) \prod_{n\ge1} 
\frac{(1-q^nw)(1-q^nw^{-1})}{(1-q^n)^2},
\end{equation}
with simple zeros at $q^n$, $n\in \Z$ and functional equations
\begin{equation}
\label{functional-equation}
\begin{split}
\text{i)}\qquad& \theta (qw) = - w^{-1}\theta (w)\\
\text{ii)}\qquad& \theta (w^{-1}) = - w^{-1}\theta (w)
\end{split}
\end{equation}
We have the following expansion
\begin{equation}\label{expansion}
\frac1{\theta (w)} = \frac1{1-w} + 
\sum_{\substack{n,m >0 \\ n\not\equiv m\mod 2}} 
(-1)^n\ q^{\frac{nm}2}\ w^{\frac{m-n-1}2}
\end{equation}
This is classical but not that well known.  For a proof see, for
example, ~\cite[Chap.VI, p. 453]{jordan}, where it is deduced from a
more general expansion due to Kronecker. Namely,
\begin{equation*}
  \frac{\theta(uv)}{\theta(u)\theta(v)}=\sum_{m,n\geq
    0}q^{mn}u^mv^n-\sum_{m,n\geq 1}q^{mn}u^{-m}v^{-n}.
\end{equation*}
(To see this set $v=u^{-\tfrac12}$ and use the functional equation
\eqref{functional-equation} to get
$$
\frac1{\theta (w)} = \frac1{1-w} + \sum_{m,n\geq 1}q^{mn}
(w^{m-\tfrac12(n+1)}-w^{m+\tfrac12(n-1)}),
$$
which is equivalent to \eqref{expansion}.) It is not hard, as was
shown to us by J. Tate, to give a direct proof
using~\eqref{functional-equation}.

{}From \eqref{expansion} we deduce, switching $q$ to $T$ and $w$ to $q$,
that
\begin{equation}\label{expansion2}
\prod_{n\ge1} \frac{(1-T^n)^2}{(1-qT^n)(1-q^{-1}T^n)} 
= 1 + \sum_{\substack{r,s>0\\ r\not\equiv s \mod 2}} 
(-1)^r T^{\frac{rs}2} \left( q^{\frac{s-r-1}2} - q^{\frac{2-r+1}2}\right)
\end{equation}
which combined with Theorem~\ref{euler-spec} gives
\begin{equation}\label{expansion-combo}
\H_{(n-1,1)} \left(\sqrt{q},\frac1{\sqrt{q}}\right) 
= \sum_{\substack{rs=2n\\ r\not\equiv s\mod2}} 
(-1)^r \left( q^{\frac{s-r-1}2} - q^{\frac{2-r+1}2}\right)
\end{equation}

We compute the logarithm of the left hand side of \eqref{expansion2}
and get
\begin{equation*}
\sum_{m,n\ge1} (q^m + q^{-m} -2) \frac{T^{mn}}{m}
\end{equation*}
Applying $(q\frac{d}{dq})^k$ and then setting $q=1$ we obtain
\begin{equation*}
\sum_{m,n\ge1} (m^k + (-m)^k) \frac{T^{mn}}{m},
\end{equation*}
which vanishes identically if $k$ is odd.
For $k$ even, it equals
$$
2 \sum_{n\ge1} \sum_{d\mid n} d^{k-1}\ T^n.
$$
Comparing with~\eqref{Eisenstein-defn} we see that this series equals
$2G_k$, up to the constant term.

Note that if $q=e^u$ then 
$$
q\frac{d}{dq} = \frac{d}{du}\ ,\qquad q=1 \leftrightarrow u=0.
$$
Hence ,
$$
\log \bigg( 1+\sum_{n\ge1} \H_{(n-1,1)} 
( e^{u/2}, e^{-u/2})T^n\bigg) = \sum_{\substack{k\ge2 \\ \text{even}}}
\left( 2G_k + \frac{B_k}{k}\right) \frac{u^k}{k!}.
$$
On the other hand, it is easy to check that
$$
u\exp \bigg( \sum_{k\ge2}\frac{B_k}{k}\ \frac{u^k}{k!}\bigg) 
= e^{u/2} - e^{-u/2}
$$
($B_k=0$ if $k>1$ is odd.)  
This proves the claim.
\end{proof}

\section{Connectedness of character varieties}

\subsection{The main result}

Let $\muhat$ be a multi-partition $(\mu^1,\dots,\mu^k)$ of $n$ and let
$\M_{\muhat}$ be a genus $g$ generic character variety of type
$\muhat$ as in \S \ref{char}. 


\begin{theorem}The character variety $\M_\muhat$ is connected (if not empty). 
\label{connectedness}\end{theorem}

Let us now explain the strategy of the proof.

We first need the following lemma.

\begin{lemma}
If $\M_\muhat$ is not empty, its number of connected components equals the constant term in $E(\M_\muhat;q)$.
\end{lemma}

\begin{proof}The number of connected components of
$\M_\muhat$ is ${\rm dim}\hspace{.05cm}H^0(\M_\muhat,\C)$ which is
also equal to the mixed Hodge number $h^{0,0;0}(\M_\muhat)$. 

Poincar\'e duality implies that

$$
h^{i,j;k}(\M_\muhat)=h_c^{d_\muhat-i,d_\muhat-j;2d_\muhat-k}(\M_\muhat).
$$
From Formula (\ref{curious}) we thus have

$$
E(\M_\muhat;q)=\sum_i\left(\sum_k(-1)^kh^{i,i;k}(\M_\muhat)\right)q^i,
$$

On the
other hand the mixed Hodge numbers $h^{i,j;k}(X)$ of any complex
non-singular variety $X$ are zero if $(i,j,k)\notin
\{(i,j,k)|\hspace{.05cm} i\leq k,j\leq k,k\leq i+j\}$, see
\cite{Del1}. Hence $h^{0,0;k}(\M_\muhat)=0$ if $k>0$. 

We thus deduce that the constant term of $E(\M_\muhat;q)$
is $h^{0,0;0}(\M_\muhat)$.
\end{proof}

From the above lemma and Formula (\ref{mainresult})  we are reduced to prove that the
coefficient of the lowest power $q^{-\frac{d_\muhat}{2}}$ of $q$ in
$\mathbb{H}_\muhat(\sqrt{q},1/\sqrt{q})$ is equal to $1$.  

The
strategy to prove this goes in two steps. First, \ref{step-1} we
analyze the lowest power of $q$ in $\calA_{\lambda\muhat}(q)$, where
$$
\Omega\left(\sqrt{q},1/\sqrt{q}\right) = \sum_{\lambda, \muhat}
\calA_{\lambda\muhat}(q)\, m_\muhat.
$$
Then \ref{step-2} we see how these combine in $\Log
\left(\Omega\left(\sqrt{q},1/\sqrt{q}\right)\right)$. In both case,
Lemma~\ref{nrm-ineq} and Lemma~\ref{ineq-1}, we will use in an
essential way the inequality of \S \ref{appendix}. Though very similar, the
relation between the partitions $\nu^p$ in these lemmas and the matrix
of numbers $x_{i,j}$ in \S \ref{appendix} is dual to each other (the $\nu^p$
appear as rows in one and columns in the other).

\subsection{Preliminaries}
\label{delta-non-neg}


For a multi-partition $\muhat\in \left(\calP_n\right)^k$ we define
\begin{equation}
\label{Delta-defn}
\Delta(\muhat):=\tfrac12 d_\muhat-1= \tfrac12(2g-2+k)n^2- \tfrac12
\sum_{i,j} \left(\mu_j^i\right)^2.
\end{equation}


\begin{remark}
  Note that when $g=0$ the quantity $-2\Delta(\muhat)$ is Katz's {\it
    index of rigidity} of a solution to $X_1\cdots X_k=I$ with $X_i\in
  \calC_i$ (see for example \cite{kostov2}[p. 91]).
\end{remark}

From $\muhat$ we define as above Theorem \ref{MA} a comet-shaped quiver $\Gamma=\Gamma_\muhat$ as well as a dimension vector $\v=\v_\muhat$ of $\Gamma$. We denote by $I$ the set of vertices of $\Gamma$ and by $\Omega$ the set of arrows. Recall that $\muhat$ and $\v$ are
linearly related ($v_0=n$ and $v_{[i,j]}=n-\sum_{r=1}^j \mu^i_r$ for
$j>1$ and conversely, $\mu_1^i=n-v_{[i,1]}$ and
$\mu_j^i=v_{[i,j-1]}-v_{[i,j]}$ for $j>1$). Hence $\Delta$ yields an
integral-valued quadratic from on $\Z^I$. Let $(\cdot,\cdot)$ be the
associated bilinear form on $\Z^I$ so that
\begin{equation}
\label{pairing-defn}
(\v,\v)=2\Delta(\muhat).
\end{equation}
Let $\e_0$ and $\e_{[i,j]}$ be the fundamental roots of $\Gamma$
(vectors in $\Z^I$ with all zero coordinates except for a $1$ at the
indicated vertex). We find that
$$
(\e_0,\e_0)= 2g-2, \qquad (\e_{[i,j]},\e_{[i,j]})=-2, \qquad
(\e_0,\e_{[i,1]})=1 \qquad (\e_{[i,j]},\e_{[i,j+1]})=1,
$$
for $ i=1,2,\ldots,k,  j=1,2,\ldots, s_i-1$ and all other
pairings are zero. In other words, $\Delta$ is the negative of the
Tits quadratic form of $\Gamma$ (with the natural orientation of all
edges pointing away from the central vertex).

With this notation we define
\begin{equation}
\label{delta-defn}
\delta=\delta(\muhat):=(\e_0,\v)=(2g-2+k)n-\sum_{i=1}^k\mu_1^i.
\end{equation}

\begin{remark}
In the case of $g=0$ the quantity $\delta$ is called the {\it defect}
by Simpson (see \cite[p.12]{simpson1}).
\end{remark}
Note that $\delta\geq (2g-2)n$ is non-negative unless $g=0$.  On the
other hand,
\begin{equation}
\label{root-ineq}
(\e_{[i,j]},\v)=\mu_j^i-\mu_{j+1}^i\geq 0.
\end{equation}
We now follow the terminology of \cite{kacconj}.
\begin{lemma}
\label{fund-set}
  The dimension vector $\v$ is in the fundamental set of imaginary
  roots of $\Gamma$ if and only if $\delta(\muhat) \geq 0$.
\end{lemma}
\begin{proof}
  Note that $v_{[i,j]}>0$ if $j<l(\mu^i)$ and $v_{[i,j]}=0$ for $j\geq
  l(\mu^i)$; since $n>0$ the support of $\v$ is then connected.  We
  already have $(\e_{[i,j]},\v)\geq 0$ by \eqref{root-ineq}, hence
  $\v$ is in the fundamental set of imaginary roots of $\Gamma$ if
  and only if $\delta \geq 0$ (see \cite{kacconj}).
\end{proof}
For a partition $\mu\in \calP_n$ we define
$$
\sigma(\mu):=n\mu_1-\sum_j\mu_j^2
$$
and extend to a multipartition $\muhat \in \left(\calP_n\right)^k$ by
$$
\sigma(\muhat):=\sum_{i=1}^k\sigma(\mu^i).
$$
\begin{remark}
Again for $g=0$ this is called the {\it superdefect} by Simpson.
\end{remark}

We say that $\mu\in \calP_n$ is {\it rectangular} if and only if all
of its (non-zero) parts are equal, i.e., $\mu=(t^{n/t})$ for some
$t\mid n$. We extend this to multi-partitions: $\muhat=
(\mu^1,\ldots,\mu^k) \in \left(\calP_n\right)^k$ is rectangular if
each $\mu^i$ is (the $\mu^i$'s are not required to be of the same
length). Note that $\muhat$ is rectangular if and only if the
associated dimension vector $\v$ satisfies $(\e_{[i,j]},\v)= 0$ for
all $[i,j]$ by \eqref{root-ineq}.

\begin{lemma}
\label{sigma-ineq}
For $\muhat \in \left(\calP_n\right)^k$ we have
$$
\sigma(\muhat) \geq 0
$$
with equality if and only if $\muhat$ is rectangular.
\end{lemma}
\begin{proof}
For any $\mu\in \calP_n$ we have $n\mu_1=\mu_1\sum_j\mu_j \geq
\sum_j\mu_j^2$ and equality holds if and only if $\mu_1=\mu_j$.
\end{proof}
Since 
\begin{equation}
\label{Delta-sigma}
2\Delta(\muhat)=n\,\delta(\muhat)+\sigma(\muhat)
\end{equation}
 we find that
\begin{equation}
\label{dim-ineq}
d_\muhat  \geq n\,\delta(\muhat)+2
\end{equation}
and in particular $d_\muhat \geq 2$ if $\delta(\muhat)\geq 0$.

If $\Gamma$ is affine it is known that the positive imaginary
roots are of the form $t\v^*$ for an integer $t\geq 1$ and some
$\v^*$. We will call $\v^*$ the {\it basic positive imaginary root} of
$\Gamma$.  The affine star-shaped quivers are given in the table
below; their basic positive imaginary root is the dimension vector
associated to the indicated multi-partition $\muhat^*$. These
$\muhat^*$, and hence also any scaled version $t\muhat^*$ for $t\geq
1$, are rectangular. Moreover, $\Delta(\muhat^*)=0$ and in
fact, $\muhat^*$ generates the one-dimensional radical of the
quadratic form $\Delta$ so that $\Delta(\muhat^*,\nuhat)=0$ for all
$\nuhat$.

\begin{proposition}
\label{affine-descrip}
Suppose that $\muhat = (\mu^1,\ldots,\mu^k) \in
\left(\calP_n\right)^k$ has $\delta(\muhat)\geq 0$. Then $d_\muhat=2$
if and only if $\Gamma$ is of affine type, i.e., $\Gamma$ is either
the Jordan quiver $J$ (one loop on one vertex), $\tilde D_4,\tilde
E_6,\tilde E_7$ or $\tilde E_8$, and $\muhat=t\muhat^*$ (all parts
scaled by $t$) for some $t\geq 1$, where $\muhat^*$, given in the
table below, corresponds to the basic imaginary root of $\Gamma$.
\end{proposition}
\begin{proof}
  By \eqref{Delta-sigma} and Lemma~\ref{sigma-ineq} $d_\muhat=2$ when
  $\delta(\muhat)\geq 0$ if and only if $\delta(\muhat)=0$ and
  $\muhat$ is rectangular.  As we observed above
  $\delta(\muhat)\geq (2g-2)n$. Hence if $\delta(\muhat)=0$ then $g=1$
  or $g=0$. If $g=1$ then necessarily $\mu^i=(n)$ and $\Gamma$ is the
  Jordan quiver $J$.

  If $g=0$ then $\delta=0$ is equivalent to the equation
\begin{equation}
\label{affine-eqn}
\sum_{i=1}^k\frac{1}{l_i}=k-2,
\end{equation}
where $l_i:=n/t_i$ is the length of $\mu^i=(t_i^{n/t_i})$.  In solving
this equation, any term with $l_i=1$ can be ignored.  It is elementary
to find all of its solutions; they correspond to the cases
$\Gamma=\tilde D_4,\tilde E_6,\tilde E_7$ or $\tilde E_8$.

We summarize the results in the following table 
\begin{equation}
\label{table}
\begin{array}{|c|c|c|c|}
\hline
\Gamma & l_i& n & \muhat^* \\
\hline
J & (1) & 1& (1) \\
\tilde{D}_4 &  (2,2,2,2) & 2 &(1,1),\quad (1,1),\quad (1,1),\quad (1,1)\\
\tilde{E}_6&  (3,3,3) &  3 & (1,1,1),\quad (1,1,1),\quad (1,1,1)\\
\tilde{E}_7 & (2,4,4) & 4 &(2,2),\quad (1,1,1,1),\quad (1,1,1,1)\\
\tilde{E}_8 & (2,3,6) & 6 & (3,3),\quad (2,2,2),\quad (1,1,1,1,1,1)\\
\hline
\end{array}
\end{equation}
where we listed the cases with smallest possible positive values of
$n$ and $k$ and the corresponding multi-partition~$\muhat^*$.
\end{proof}
\noindent
Proposition~\ref{affine-descrip} is due to Kostov, see for example
\cite[p.14]{simpson1}.

We will need the following result about $\Delta$.
\begin{proposition}
\label{Delta-ineq}
Let $\muhat\in \left(\calP_n\right)^k$ and
$\nuhat^p=(\nu^{1,p},\ldots,\nu^{k,p})\in \left(\calP_{n_p}\right)^k$
for $p=1,\ldots, s$ be non-zero multi-partitions such that up to
permutations of the parts of $\nu^{i,p}$ we have
$$
\mu^i=\sum_{p=1}^s \nu^{i,p}, \qquad \qquad i=1,\ldots,k.
$$
 Assume that $\delta(\muhat)\geq 0$. Then
$$
\sum_{p=1}^s\Delta(\nuhat^p)\leq\Delta(\muhat).
$$
Equality holds if  and only if

(i) $s=1$ and $\muhat=\nuhat^1$.

or

(ii) $\Gamma$ is affine and $\muhat,\nuhat^i,\ldots,\nuhat^s$
correspond to positive imaginary roots.

\end{proposition}

We start with the following. For partitions $\mu,\nu$ define
$$
\nrm_\mu(\nu):=\mu_1 |\nu|^2 -|\mu|\sum_i \nu_i^2.
$$
Note that $\nrm_\mu(\mu)=|\mu|\,\sigma(\mu)$.
\begin{lemma}
\label{nrm-ineq}
Let $\nu^1,\ldots,\nu^s$ and $\mu$ be non-zero partitions such that up
to permutation of the parts of each $\nu^p$ we have
$\sum_{p=1}^s\nu^p=\mu.$ Then
$$
\sum_{p=1}^s \nrm_\mu(\nu^p)\leq \nrm_\mu(\mu).
$$
Equality holds if and only if:

(i) $s=1$ and $\mu=\nu^1$.

or

(ii) $\nu^1,\ldots,\nu^s$ and $\mu$ all are rectangular of the
same length.

\end{lemma}
\begin{proof}
  This is just a restatement of the inequality of \S \ref{appendix} with
  $x_{i,k}=\nu_{\sigma_k(i)}^k$, for the appropriate permutations
  $\sigma_k$, where $1\leq i \leq l(\mu),1\leq k \leq s$.
\end{proof}
\begin{lemma}
\label{sigma-rectang}
If  the partitions $\mu,\nu$ are rectangular of the same length
then 
$$
\sigma_\mu(\nu)=0.
$$
\end{lemma}
\begin{proof}
  Direct calculation.
\end{proof}

\begin{proof}[Proof of Proposition~\ref{Delta-ineq}]
From the definition \eqref{Delta-defn} we get 
$$
2n\Delta(\muhat)= \delta(\muhat)n^2+\sum_{i=1}^k\nrm_{\mu^i}(\mu^i)
$$
and similarly 
$$
2n\Delta(\nuhat^p)=\delta(\muhat)n_p^2 +
\sum_{i=1}^k\nrm_{\mu^i}(\nu^{i,p}), \qquad \qquad p=1,\ldots,s
$$
 hence
$$
2n\sum_{p=1}^s\Delta(\nuhat^p)=\delta(\muhat)\sum_{p=1}^sn_p^2 +
\sum_{i=1}^k\sum_{p=1}^s\nrm_{\mu^i}(\nu^{i,p}).
$$
Since $n=\sum_{p=1}^sn_p$ and $\delta(\muhat)\geq 0$ we get from
Lemma~\ref{nrm-ineq} that 
$$
\sum_{p=1}^s\Delta(\nuhat^p)\leq\Delta(\muhat)
$$
as claimed.

Clearly, equality cannot occur if $\delta(\muhat)>0$ and $s>1$. If
$\delta(\muhat)=0$ and $s>1$ it follows from
Lemmas~\ref{nrm-ineq},~\ref{sigma-rectang} and \eqref{Delta-sigma}
that $\Delta(\muhat)=\Delta(\nuhat^p)=0$ for $p=1,2,\ldots,s$. 
Now (ii) is a consequence of Proposition~\ref{affine-descrip}.
\end{proof}

\subsection{Proof of Theorem~\ref{connectedness}}

\subsubsection{Step I}
\label{step-1}

Let 
\begin{equation}
\label{calA-defn}
\calA_{\lambda\muhat}(q):= 
q^{(1-g)|\lambda|}\left(q^{-n(\lambda)} 
  H_{\lambda}(q)\right)^{2g+k-2}
\prod_{i=1}^k\left\langle
  h_{\mu^i}(\x_i),
  s_\lambda(\x_i\y)\right\rangle,
\end{equation}
so that by Lemma \ref{specializ} 
$$
\Omega\left(\sqrt{q},1/\sqrt{q}\right) = \sum_{\lambda, \muhat}
\calA_{\lambda\muhat}(q)\, m_\muhat.
$$It is easy to verify that
$\calA_{\lambda\muhat}$ is in $\Q(q)$.

For a non-zero rational function $\calA\in \Q(q)$ we let
$v_q\left(\calA\right)\in \Z$ be its valuation at $q$.  We will see
shortly that $\calA_{\lambda\muhat}$ is nonzero for all $\lambda,
\muhat$; let $v(\lambda) :=v_q\left(\calA_{\lambda\muhat}(q)\right)$.
The first main step toward the proof of the connectedness is the
following theorem.


\begin{theorem}
\label{minim}
Let $\muhat=(\mu^1,\mu^2,\ldots,\mu^k)\in {\P_n}^k$ with
$\delta(\muhat) \geq 0$. Then
 
i) The minimum value of $v(\lambda)$ as
  $\lambda$ runs over the set of partitions of size $n$, 
is  
$$
v((1^n))=-\Delta(\muhat).
$$ 

ii) There are two cases as to where this minimum occurs. 

\medskip Case I: The quiver $\Gamma$ is affine and the dimension
vector associated to $\muhat$ is a positive imaginary root $t\v^*$ for
some $t\mid n$.  In this case, the minimum is reached at all
partitions $\lambda$ which are the union of $n/t$ copies of any
$\lambda_0\in\calP_t$.

\medskip Case II: Otherwise, the minimum occurs only at
$\lambda=(1^n)$.
 \label{step1}
\end{theorem} 

Before proving the theorem we need some preliminary results. 
\begin{lemma}
$\langle h_{\mu}(\x),s_\lambda(\x\y)\rangle$ is non-zero for all
$\lambda$ and $\mu$. 
\end{lemma}
\begin{proof}
  We have $s_\lambda(\x\y)=\sum_\nu K_{\lambda\nu}m_\nu(\x\y)$
  \cite[I 6 p.101]{macdonald}  and $m_\nu(\x\y)=\sum_\mu
  C_{\nu\mu}(\y)\,m_\mu(\x)$ for some $C_{\nu\mu}(\y)$. Hence
\begin{equation}
\label{m-coeff}
\langle
  h_\mu(\x),s_\lambda(\x\y)\rangle=\sum_\nu
  K_{\lambda\nu}C_{\nu\mu}(\y).
\end{equation}
For any  set of variables  $\x\y=\{x_iy_j\}_{1\leq i,1\leq 
  j}$ we have
\begin{equation}
\label{C-fmla}
C_{\nu\mu}(\y)=\sum m_{\rho^1}(\y)\cdots m_{\rho^r}(\y),
\end{equation}
where the sum is over all partitions $\rho^1,\ldots, \rho^r$ such that
$|\rho^p|=\mu_p$ and $\rho^1\cup \cdots \cup \rho^r=\nu$. In
particular the coefficients of $C_{\nu\mu}(\y)$ as power series in $q$
are non-negative. We can take, for example, $\rho^p=(1^{\mu_p})$ and
then $\nu=(1^n)$. Since $K_{\lambda\nu}\geq 0$ \cite[I
(6.4)]{macdonald} for any $\lambda,\nu$ and
$K_{\lambda,(1^n)}=n!/h_\lambda$ \cite[I 6 ex. 2]{macdonald}, with
$h_\lambda=\prod_{s\in \lambda}h(s)$ the product of the hook lengths,
we see that $\langle h_\mu(\x),s_\lambda(\x\y)\rangle$ is non-zero and
our claim follows.
\end{proof}
In particular $\calA_{\lambda\muhat}$ is non-zero for all $\lambda$
and $\muhat$.  Define
\begin{equation}
\label{v-1-defn}
v(\lambda,\mu):=v_q\left(\langle
h_{\mu}(\x),s_\lambda(\x\y)\rangle\right).
\end{equation}
\begin{lemma}
\label{v-lambda}
 We  have 
$$
-v(\lambda)=(2g-2+k) n(\lambda)+(g-1)n
-\sum_{i=1}^kv(\lambda,\mu^i).
$$
 \end{lemma}
\begin{proof}
Straightforward.
\end{proof}

\begin{lemma}
\label{v-fmla-lemma}
 For $\mu=(\mu_1,\mu_2,\ldots,\mu_r)\in \calP_n$ we have
\begin{equation}
\label{v-fmla}
v(\lambda,\mu)= \min
\{n(\rho^1)+\cdots+n(\rho^r)\,\,|\,\, |\rho^p|=\mu_p, \, 
\cup_p\rho^p \unlhd\lambda\}.
\end{equation}
\end{lemma}
\begin{proof}
  For $C_{\nu\mu}(\y)$ non-zero let
  $v_m(\nu,\mu):=v_q\left(C_{\nu\mu}(\y)\right)$.  When $y_i=q^{i-1}$
  we have $v_q(m_\rho(\y))= n(\rho)$ for any partition $\rho$. Hence
  by \eqref{C-fmla}
$$
v_m(\nu,\mu)= \min \{n(\rho^1)+\cdots+n(\rho^r)\,\,|\,\, |\rho^p|=\mu_p, \, 
\cup_p\rho^p=\nu\}.
$$
Since $K_{\lambda\nu}\geq 0$ for any $\lambda,\nu$, $K_{\lambda\nu}>0$
if and only if $\nu\unlhd \lambda$ \cite[Ex 2, p.26]{fulton}, and the
coefficients of $C_{\nu\mu}(\y)$ are non-negative, our claim follows
from \eqref{m-coeff}.
\end{proof}

For example, if $\lambda=(1^n)$ then necessarily $\rho^p=(1^{\mu_p})$
and hence $\rho^1\cup \cdots \cup \rho^r=\lambda$. We have then
\begin{equation}
\label{v-example}
v\left((1^n\right),\mu)=\sum_{p=1}^r\binom{\mu_p} 2 
= -\tfrac12 n +\tfrac12 \sum_{p=1}^r \mu_p^2.
\end{equation}
Similarly,
\begin{equation}
\label{v-example-1}
v(\lambda,(n))=n(\lambda)
\end{equation}
by the next lemma.

\begin{lemma}
\label{n-ineq}
If $\beta\unlhd\alpha$ then $n(\alpha)\leq n(\beta)$ with equality if
and only if $\alpha=\beta$.
\end{lemma}
\begin{proof}
  We will use the raising operators $R_{ij}$ see \cite[I
  p.8]{macdonald}.  Consider vectors $w$ with 
  coefficients in $\Z$ and extend the function $n$ to them in the
  natural way
$$
n(w):=\sum_{i\geq 1}(i-1)w_i.
$$
Applying a raising operator $R_{ij}$, where $i<j$, has the effect
$$
n(R_{ij}w)=n(w)+i-j.
$$
Hence for any product $R$ of raising operators we have $n(Rw) < n(w)$
with equality if and only if $R$ is the identity operator. Now the
claim follows from the fact that $\beta\unlhd\alpha$ implies there
exist such and $R$ with $\alpha=R\beta$.
\end{proof}

Recall \cite[(1.6)]{macdonald} that for any partition $\lambda$ we have
$\langle\lambda,\lambda\rangle=2n(\lambda)+|\lambda|=\sum_i(\lambda'_i)^2$,
where $\lambda'=(\lambda'_1,\lambda'_2,\dots)$ is the dual
partition. Note also that $(\lambda\cup\mu)'=\lambda'+\mu'$.  Define
$$
\|\lambda\|:=\sqrt{\langle\lambda',\lambda'\rangle}=\sqrt{\sum_i\lambda_i^2}.
$$
The following inequality is a particular case of the theorem of
\S \ref{appendix}. 
\begin{lemma}
\label{ineq-1}
  Fix $\mu=(\mu_1,\mu_2,\ldots,\mu_r)\in \calP_n$. Then for every
  $(\nu^1,...,\nu^r)\in\calP_{\mu_1}\times\cdots\times\calP_{\mu_r}$
  we have
\begin{equation}
\mu_1\left\|\sum_p\nu^p\right\|^2-n\sum_p\left\|\nu^p\right\|^2\leq
\mu_1n^2-n\left\|\mu\right\|^2.
\label{cauchy-ineq}
\end{equation}
Moreover, equality holds in \eqref{cauchy-ineq} if and only if either:

(i)  The partition $\mu$ is rectangular and all partitions
$\nu^p$ are equal. 

or

(ii) For each $p=1,2,\ldots,r$  we have $\nu^p=(\mu_p)$.
\end{lemma}
\begin{proof}
  Our claim is a consequence of the theorem of \S \ref{appendix}. Taking
  $x_{ps}=\nu^p_s$ we have $c_p:=\sum_sx_{ps}=\sum_s\nu^p_s=\mu_p$ and
  $c:=\max_pc_p=\mu_1$.
\end{proof}

The following fact will be crucial for the proof of connectedness.
\begin{proposition}
\label{maxima}
 For a fixed
$\mu=(\mu_1,\mu_2,\ldots,\mu_r)\in \calP_n$ we have
$$
\mu_1n(\lambda)-nv(\lambda,\mu) \leq \mu_1n^2-n\left\|\mu\right\|^2,
\qquad \lambda \in \calP_n.
$$
Equality holds only at $\lambda=(1^n)$ unless $\mu$ is rectangular
$\mu=(t^{n/t})$, in which case it also holds when $\lambda$ is the union
of $n/t$ copies of any $\lambda_0\in\calP_t$.
 \label{opti}
\end{proposition}

\begin{proof} 
  Given $\nu \unlhd \lambda$ write $\mu_1n(\lambda)-nv(\lambda,\mu)$
  as
\begin{equation}
\label{first-step}
\mu_1n(\lambda)-nv(\lambda,\mu)= 
\mu_1 (n(\lambda)-n(\nu)) +
\mu_1n(\nu)-nv(\lambda,\mu)
\end{equation}
By Lemma~\ref{n-ineq} the first term is
non-negative. Hence
$$
\mu_1n(\lambda)-nv(\lambda,\mu)\leq
\mu_1n(\nu)-nv(\lambda,\mu), \qquad \qquad \nu\unlhd \lambda.
$$
Combinining this with \eqref{v-fmla} yields
\begin{equation}
\label{ineq}
\max_{|\lambda|=n}\left[\mu_1n(\lambda)-nv(\lambda,\mu)\right]\leq
\max_{|\rho^p|=\mu_p}\left[\mu_1n(\rho^1\cup\rho^2\cup\cdots\cup\rho^r)-
(n(\rho^1)+\cdots+n(\rho^r))n  \right].  
\end{equation}

\noindent
Take $\nu^p$ to be the dual of $\rho^p$ for $p=1,2,\ldots,r$. Then the
right hand side of \eqref{ineq} is precisely
$$
\mu_1\left\|\sum_p\nu^p\right\|^2-n\sum_p\left\|\nu^p\right\|^2,
$$
which by Lemma~\ref{ineq-1} is bounded above by
$\mu_1n^2-n\left\|\mu\right\|^2$ with equality only where either
$\rho^p=(1^{\mu_p})$ (case (ii)) or all $\rho^p$ are equal and
$\mu=(t^{n/t})$ for some $t$ (case (i)).

Combining this with Lemma~\ref{n-ineq} we see that to obtain the
maximum of the left hand side of \eqref{ineq} we must also have
$\rho^1\cup \cdots\cup\rho^r=\lambda$. In case (i) then, $\lambda$ is
the union of $n/t$ copies of $\lambda_0$, the common value of
$\rho^p$, and in case (ii), $\lambda=(1^n)$. 
\end{proof}

\begin{proof}[Proof of Theorem \ref{step1}]
  We first prove (ii).  Using Lemma~\ref{v-lambda} we
  have
 \begin{equation}
   -v(\lambda)=(2g-2+k) n(\lambda) +(g-1)n -\sum_{i=1}^kv(\lambda,\mu^i)=
   \frac{\delta}{n}n(\lambda)+(g -1)n
   +\frac1n\sum_{i=1}^k\left[\mu_1^in(\lambda)-nv(\lambda,\mu^i)\right].
\label{vlambda}
\end{equation}
The terms $n(\lambda)$ and
$\sum_{i=1}^n\left[\mu_1^in(\lambda)-nv(\lambda,\mu^i)\right]$ are all
maximal at $\lambda=(1^n)$ (the last by Proposition \ref{opti}). Hence
$-v(\lambda)$ is also maximal at $(1^n)$, since $\delta\geq 0$. Now
$n(\lambda)$ has a unique maximum at $(1^n)$ by Lemma~\ref{n-ineq},
hence $-v(\lambda)$ reaches its maximum at other partitions if and
only if $\delta=0$ and for each $i$ we have $\mu^i=(t_i^{n/t_i})$ for
some positive integer $t_i\mid n$ (again by Proposition \ref{opti}).
In this case the maximum occurs only for $\lambda$ the union of $n/t$
copies of a partition $\lambda_0\in \calP_t$, where $t=\gcd t_i$.  Now
(ii) follows from Proposition~\ref{affine-descrip}.

To prove (i) we use Lemma~\ref{v-lambda} and \eqref{v-example} and find
that $v((1^n))=-\Delta(\muhat)$ as claimed.
\end{proof}

\begin{lemma}
  Let $\muhat=(\mu^1,\mu^2,\ldots,\mu^k)\in {\P_n}^k$ with
  $\delta(\muhat) \geq 0$. Suppose that $v(\lambda)$ is minimal.  Then
  the coefficient of $q^{v(\lambda)}$ in $\calA_{\lambda \muhat}$ is
  $1$.
\label{lowest}
\end{lemma}

\begin{proof}
  We use the notation of the proof of Lemma~\ref{v-fmla-lemma}. Note
  that the coefficient of the lowest power of $q$ in
  $\calH_\lambda(\sqrt q,1/\sqrt q)
  \left(q^{-n(\lambda)}H_\lambda(q)\right)^k$ is $1$ (see
  \eqref{H-specializ}).  Also, the coefficient of the lowest power of
  $q$ in each $m_{\lambda}(\y)$ is always $1$; hence so is the
  coefficient of the lowest power of $q$ in $C_{\nu\mu}(\y)$.

  In the course of the proof of Proposition~\ref{maxima} we found that
  when $v(\lambda)$ is minimal, and $\rho^1,\ldots,\rho^r$ achieve the
  minimum in the right hand side of \eqref{v-fmla}, then
  $\lambda=\rho^1\cup\cdots\cup\rho^r$.  Hence by
  Lemma~\ref{v-fmla-lemma}, the coefficient of the lowest power of
  $q$ in $\langle
  h_{\mu}(\x),s_\lambda(\x\y)\rangle=\sum_{\nu\unlhd\lambda}
  K_{\lambda\nu}C_{\nu\mu}(\y)$ equals the coefficient of the lowest
  power of $q$ in
  $K_{\lambda\lambda}C_{\lambda\mu}(\y)=C_{\lambda\mu}(\y)$ which we
  just saw is $1$. This completes the proof.
\end{proof}


\subsubsection{Leading terms of $\Log\,\Omega$}
\label{step-2}
We now proceed to the second step in the proof of connectedness where
we analyze the smallest power of $q$ in the coefficients of
$\Log\left(\Omega\left(\sqrt{q},1/\sqrt{q}\right)\right)$.  Write
\begin{equation}
\Omega\left(\sqrt{q},1/\sqrt{q}\right)=\sum_\muhat
  P_\muhat(q)\,m_\muhat
\label{pmu}
\end{equation}
with $P_\muhat(q):=\sum_\lambda\calA_{\lambda\muhat}$ and
$\calA_{\lambda \muhat}$ as in \eqref{calA-defn}.  

Then by Lemma \ref{Log-w} we have 
$$
\Log\left(\Omega\left(\sqrt{q},1/\sqrt{q}\right)\right)
=\sum_\omhat C_\omhat^0 P_\omhat(q)\,m_\omhat(q)
$$
where $\omhat$ runs over {\it multi-types}
$(d_1,\omhat^1)\cdots(d_s,\omhat^s)$ with $\omhat^p\in(\calP_{n_p})^k$
and
$P_\omhat(q):=\prod_pP_{\omhat^p}(q^{d_p}),
m_\omhat(\x):=\prod_pm_{\omhat^p}(\x^{d_p})$..

Now if we let $\gamma_{\muhat\omhat}:=\langle
m_\omhat,h_\muhat\rangle$ then we have
$$
\H_\muhat\left(\sqrt{q},1/\sqrt{q}\right)
=\frac{(q-1)^2}{q}\left(\sum_{\omhat\in\mathbf{T}^k}
  C_{\omhat}^0P_\omhat(q)\gamma_{\muhat\omhat}\right). 
$$  
By Theorem \ref{step1},
$v_q\left(P_\omhat(q)\right)=-d\sum_{p=1}^s\Delta(\omhat^p)$ for a
multi-type $\omhat=(d,\omhat^1)\cdots(d,\omhat^s)$.

\begin{lemma}
\label{perm-sum-lemma}
Let $\nu^1,\ldots,\nu^s$ be partitions. Then
$$
\langle m_{\nu^1}\cdots m_{\nu^s}, h_\mu\rangle  \neq 0
$$
if and only if $\mu=\nu^1+\cdots+\nu^s$ up to permutation of the parts
of each $\nu^p$ for $p=1,\ldots,s$.
\end{lemma}
\begin{proof}
It follows immediately from the definition of the monomial symmetric
function.
\end{proof}

Let $\v$ be the dimension vector
associated to $\muhat$. 

\begin{theorem}
\label{connectedness1}
 If $\v$ is in the fundamental set of imaginary roots
  of $\Gamma$ then the character variety $\M_\muhat$ is non-empty and
  connected.
 \end{theorem}

\begin{proof}Assume $\v$ is in the fundamental set of roots
of $\Gamma$. By Lemma~\ref{fund-set} this is equivalent to
$\delta(\muhat)\geq 0$.

  Note that $m_\nu(\x^d)=m_{d\nu}(\x)$ for any partition $\nu$ and
  positive integer $d$. Suppose
  $\omhat=(d,\omhat^1)\cdots(d,\omhat^s)$ is a multi-type for which
  $\gamma_{\muhat\omhat}$ is non-zero. Let $\nuhat^p=d\omhat^p$ for
  $p=1,\ldots, s$ (scale every part by $d$). These multi-partitions
  are then exactly in the hypothesis of Proposition~\ref{Delta-ineq}
  by Lemma~\ref{perm-sum-lemma}. Hence
\begin{equation}
\label{conn-ineq}
 d\sum_{p=1}^s
  \Delta(\omhat^p)\leq d^2\sum_{p=1}^s
  \Delta(\omhat^p)=\sum_{p=1}^s\Delta(\nuhat^p) \leq\Delta(\muhat).
\end{equation}
Suppose $\Gamma$ is not affine. Then by Proposition~\ref{Delta-ineq}
we have equality of the endpoints in \eqref{conn-ineq} if and only if
$s=1$, $\nuhat^1=\muhat$ and $d=1$, in other words, if and only if
$\omhat=(1,\muhat)$. Hence, since $C_{(1,\,\muhat)}^0=1$, the
coefficient of the lowest power of $q$ in
$\H_\muhat\left(\sqrt{q},1/\sqrt{q}\right)$ equals the coefficient of
the lowest power of $q$ in $P_\muhat(q)$ which is $1$ by
Lemma~\ref{lowest} and Theorem~\ref{minim}, Case II. This proves our
claim in this case.

Suppose now $\Gamma$ is affine. Then by Proposition~\ref{Delta-ineq}
we have equality of the endpoints in \eqref{conn-ineq} if and only if
$\muhat=t\muhat^*$ and $\omhat=(1,t_1\muhat^*),\ldots,(1,t_s\muhat^*)$
for a partition $(t_1,t_2,\ldots, t_s)$ of $t$ and $d=1$.  Combining
this with Lemma~\ref{lowest} and Theorem~\ref{minim}, Case I we see
that the lowest order terms in $q$ in
$\Log\left(\Omega\left(\sqrt{q},1/\sqrt{q}\right)\right)$ are
$$
L:=\sum C^0_\omhat  p(t_1)\cdots
p(t_s) \,m_{t\muhat^*},
$$
where the sum is over types $\omhat$ as above.  Comparison with
Euler's formula 
$$
\Log\left(\sum_{n\geq 0}p(n)\,T^n\right)=\sum_{n\geq1}T^n,
$$
shows that $L$ reduces to $\sum_{t\geq 1}
\,m_{t\muhat^*}$. Hence  the coefficient of the lowest power of $q$ in
$\H_\muhat\left(\sqrt{q},1/\sqrt{q}\right)$ is also~$1$ in this case
finishing the proof.
\end{proof}

\begin{proof}[Proof of Theorem \ref{connectedness}] If $g\geq 1$, the dimension vector $\v$ is always in the fundamental
 set of imaginary roots of $\Gamma$.  If $g=0$ the character variety
 if not empty if and only if $\v$ is a strict root of $\Gamma$ and if
 $\v$ is real then $\M_\muhat$ is a point \cite[Theorem
 8.3]{crawley-par}. If $\v$ is imaginary then it can be taken by the
 Weyl group to some $\v'$ in the fundamental set and the two
 corresponding varieties $\M_\muhat$ and $\M_{\muhat'}$ are isomorphic
 for appropriate choices of conjugacy classes \cite[Theorem 3.2, Lemma
 4.3 (ii)]{crawley-par}, hence Theorem \ref{connectedness}.

\end{proof}

\section{Appendix by Gergely~Harcos}\label{appendix}

\begin{theorem} Let $n,r$ be positive integers, and let $x_{ik}$
($1\leq i\leq n$, $1\leq k\leq r$) be arbitrary nonnegative numbers.
Let $c_i:=\sum_k x_{ik}$ and $c:=\max_i c_i$. Then we we have
\[c\sum_k\biggl(\sum_i x_{ik}\biggr)^2-\biggl(\sum_i
c_i\biggr)\biggl(\sum_{i,k}x_{ik}^2\biggr)\leq c\biggl(\sum_i
c_i\biggr)^2-\biggl(\sum_i c_i\biggr)\biggl(\sum_i c_i^2\biggr).\]
Assuming $\min_i c_i>0$, equality holds if and only if  we are in
one of the following situations

(i) $x_{ik}=x_{jk}$ for all $i,j,k$,

(ii) there exists some $l$ such that
$x_{ik}=0$ for all $i$ and all $k\neq l$.
\label{harcos}\end{theorem}

\begin{remark} The assumption $\min_i c_i>0$ does not result in any
loss of generality, because the values $i$ with $c_i=0$ can be
omitted without altering any of the sums.
\end{remark}

\begin{proof} Without loss of generality we can assume $c=c_1\geq\dots\geq
c_n$, then the inequality can be rewritten as
\[\biggl(\sum_i c_i\biggr)\biggl(\sum_j\sum_{k,l}x_{jk}x_{jl}-\sum_{j,k}x_{jk}^2\biggr)
\leq c\biggl(\sum_{i,j}\sum_{k,l}x_{ik}x_{jl}-\sum_{i,j}\sum_k
x_{ik}x_{jk}\biggr).\] Here and later $i,j$ will take values from
$\{1,\dots,n\}$ and $k,l,m$ will take values from $\{1,\dots,r\}$.
We simplify the above as
\[\biggl(\sum_i c_i\biggr)\biggl(\sum_j\sum_{\substack{{k,l}\\k\neq l}}x_{jk}x_{jl}\biggr)
\leq c\biggl(\sum_{i,j}\sum_{\substack{{k,l}\\k\neq
l}}x_{ik}x_{jl}\biggr),\] then we factor out and also utilize the
symmetry in $k,l$ to arrive at the equivalent form
\[\sum_{i,j}c_i\sum_{\substack{{k,l}\\k<l}}x_{jk}x_{jl}\leq
\sum_{i,j}c\sum_{\substack{{k,l}\\k<l}}x_{ik}x_{jl}.\] We distribute
the terms in $i,j$ on both sides as follows:
\[\sum_i c_i\sum_{\substack{{k,l}\\k<l}}x_{ik}x_{il}+
\sum_{\substack{{i,j}\\i<j}}\biggl(c_i\sum_{\substack{{k,l}\\k<l}}x_{jk}x_{jl}
+c_j\sum_{\substack{{k,l}\\k<l}}x_{ik}x_{il}\biggr)\leq \sum_i
c\sum_{\substack{{k,l}\\k<l}}x_{ik}x_{il}+
\sum_{\substack{{i,j}\\i<j}}c\sum_{\substack{{k,l}\\k<l}}(x_{ik}x_{jl}+x_{jk}x_{il}).\]
It is clear that
\[c_i\sum_{\substack{{k,l}\\k<l}}x_{ik}x_{il}\leq c\sum_{\substack{{k,l}\\k<l}}x_{ik}x_{il},\quad 1\leq i\leq n,\]
therefore it suffices to show that
\[c_i\sum_{\substack{{k,l}\\k<l}}x_{jk}x_{jl}+c_j\sum_{\substack{{k,l}\\k<l}}x_{ik}x_{il}
\leq c\sum_{\substack{{k,l}\\k<l}}(x_{ik}x_{jl}+x_{jk}x_{il}),\quad
1\leq i<j\leq n.\] We will prove this in the stronger form
\[c_i\sum_{\substack{{k,l}\\k<l}}x_{jk}x_{jl}+c_j\sum_{\substack{{k,l}\\k<l}}x_{ik}x_{il}
\leq
c_i\sum_{\substack{{k,l}\\k<l}}(x_{ik}x_{jl}+x_{jk}x_{il}),\quad
1\leq i<j\leq n.\]

We now fix $1\leq i<j\leq n$ and introduce $x_k:=x_{ik}$,
$x'_k:=x_{jk}$. Then the previous inequality reads
\[\biggl(\sum_m
x_m\biggr)\biggl(\sum_{\substack{{k,l}\\k<l}}x'_kx'_l\biggr)+
\biggl(\sum_m
x'_m\biggr)\biggl(\sum_{\substack{{k,l}\\k<l}}x_kx_l\biggr)\leq
\biggl(\sum_m
x_m\biggr)\sum_{\substack{{k,l}\\k<l}}(x_kx'_l+x'_kx_l),\] that is,
\[\sum_{\substack{{k,l,m}\\k<l}}(x_mx'_kx'_l+x_kx_lx'_m)\leq
\sum_{\substack{{k,l,m}\\k<l}}(x_kx_mx'_l+x_lx_mx'_k).\] The right
hand side equals
\begin{align*}
\sum_{\substack{{k,l,m}\\k<l}}(x_kx_mx'_l+x_lx_mx'_k)
&=\sum_{\substack{{k,l,m}\\l\neq
k}}x_kx_mx'_l=\sum_{\substack{{k,l,m}\\m\neq k}}x_kx_lx'_m
=\sum_{\substack{{k,m}\\m\neq
k}}x_k^2x'_m+\sum_{\substack{{k,l,m}\\l\neq k\\m\neq k}}x_kx_lx'_m\\
&=\sum_{\substack{{k,m}\\m\neq
k}}x_k^2x'_m+\sum_{\substack{{k,m}\\m\neq k}}x_kx_mx'_m
+\sum_{\substack{{k,l,m}\\l\neq k\\m\neq k,l}}x_kx_lx'_m\\
&=\sum_{\substack{{k,m}\\m\neq
k}}x_k^2x'_m+\sum_{\substack{{k,m}\\k<m}}x_kx_mx'_m
+\sum_{\substack{{k,m}\\m<k}}x_kx_mx'_m
+2\sum_{\substack{{k,l,m}\\k<l\\m\neq k,l}}x_kx_lx'_m\\
&=\sum_{\substack{{k,m}\\m\neq
k}}x_k^2x'_m+\sum_{\substack{{k,l}\\k<l}}x_kx_lx'_l
+\sum_{\substack{{k,l}\\k<l}}x_kx_lx'_k
+2\sum_{\substack{{k,l,m}\\k<l\\m\neq k,l}}x_kx_lx'_m\\
&=\sum_{\substack{{k,m}\\m\neq
k}}x_k^2x'_m+\sum_{\substack{{k,l,m}\\k<l}}x_kx_lx'_m
+\sum_{\substack{{k,l,m}\\k<l\\m\neq k,l}}x_kx_lx'_m,
\end{align*}
therefore it suffices to prove
\[\sum_{\substack{{k,l,m}\\k<l}}x_mx'_kx'_l\leq\sum_{\substack{{k,m}\\m\neq
k}}x_k^2x'_m+\sum_{\substack{{k,l,m}\\k<l\\m\neq k,l}}x_kx_lx'_m.\]
This is trivial if $x'_m=0$ for all $m$. Otherwise $\sum_m x'_m>0$,
hence $c_i\geq c_j$ yields
\[\lambda:=\biggl(\sum_m x_m\biggr)\biggl(\sum_m x'_m\biggr)^{-1}\geq 1.\]
Clearly, we are done if we can prove
\[\lambda^2\sum_{\substack{{k,l,m}\\k<l}}x_mx'_kx'_l\leq\lambda\sum_{\substack{{k,m}\\m\neq
k}}x_k^2x'_m+\lambda\sum_{\substack{{k,l,m}\\k<l\\m\neq
k,l}}x_kx_lx'_m.\] We introduce $\tilde x_m:=\lambda x'_m$, then
\[\sum_m\tilde x_m=\sum_m x_m,\]
and the last inequality reads
\[\sum_{\substack{{k,l,m}\\k<l}}x_m\tilde x_k\tilde x_l\leq
\sum_{\substack{{k,m}\\m\neq k}}x_k^2\tilde
x_m+\sum_{\substack{{k,l,m}\\k<l\\m\neq k,l}}x_kx_l\tilde x_m.\] By
adding equal sums to both sides this becomes
\[\sum_{\substack{{k,l,m}\\k<l}}x_m\tilde x_k\tilde
x_l+\sum_{\substack{{k,l,m}\\k<l}}x_kx_l\tilde x_m\leq
\sum_{\substack{{k,m}\\m\neq k}}x_k^2\tilde
x_m+\sum_{\substack{{k,l,m}\\k<l\\m\neq k,l}}x_kx_l\tilde
x_m+\sum_{\substack{{k,l,m}\\k<l}}x_kx_l\tilde x_m,\] which can also
be written as
\[\biggl(\sum_m x_m\biggr)\biggl(\sum_{\substack{{k,l}\\k<l}}\tilde x_k\tilde x_l\biggr)+
\biggl(\sum_m \tilde
x_m\biggr)\biggl(\sum_{\substack{{k,l}\\k<l}}x_k
x_l\biggr)\leq\sum_k x_k^2\biggl(\sum_{\substack{m\\m\neq k}}\tilde
x_m\biggr)+\sum_{\substack{{k,l}\\k<l}}x_kx_l\biggl(\sum_{\substack{m\\m\neq
k}}\tilde x_m+\sum_{\substack{m\\m\neq l}}\tilde x_m\biggr).\] The
right hand side equals
\begin{align*}
\sum_k x_k^2\biggl(\sum_{\substack{m\\m\neq k}}\tilde
x_m\biggr)+\sum_{\substack{{k,l}\\k<l}}x_kx_l\biggl(\sum_{\substack{m\\m\neq
k}}\tilde x_m+\sum_{\substack{m\\m\neq l}}\tilde x_m\biggr)&= \sum_k
x_k^2\biggl(\sum_{\substack{m\\m\neq k}}\tilde x_m\biggr)+
\sum_{\substack{{k,l}\\l<k}}x_kx_l\biggl(\sum_{\substack{m\\m\neq
l}}\tilde
x_m\biggr)+\sum_{\substack{{k,l}\\k<l}}x_kx_l\biggl(\sum_{\substack{m\\m\neq
l}}\tilde x_m\biggr)\\
&= \sum_k x_k^2\biggl(\sum_{\substack{m\\m\neq k}}\tilde x_m\biggr)+
\sum_{\substack{{k,l}\\k\neq
l}}x_kx_l\biggl(\sum_{\substack{m\\m\neq l}}\tilde x_m\biggr)\\
&=\sum_{k,l}x_kx_l\biggl(\sum_{\substack{m\\m\neq l}}\tilde
x_m\biggr)=\biggl(\sum_k
x_k\biggr)\biggl(\sum_{\substack{{m,l}\\m\neq l}}x_l\tilde
x_m\biggr),
\end{align*}
hence the previous inequality is the same as
\[\biggl(\sum_m x_m\biggr)\biggl(\sum_{\substack{{k,l}\\k<l}}\tilde x_k\tilde x_l\biggr)+
\biggl(\sum_m \tilde
x_m\biggr)\biggl(\sum_{\substack{{k,l}\\k<l}}x_k x_l\biggr)\leq
\biggl(\sum_k x_k\biggr)\biggl(\sum_{\substack{{m,l}\\m\neq
l}}x_l\tilde x_m\biggr).\] The first factors are equal and positive,
hence after renaming $m,l$ to $k,l$ when $m<l$ and to $l,k$ when
$m>l$ on the right hand side we are left with proving
\[\sum_{\substack{{k,l}\\k<l}}(\tilde x_k\tilde x_l+x_k x_l)\leq
\sum_{\substack{{k,l}\\k<l}}(\tilde x_k x_l+x_k\tilde x_l).\] This
can be written in the elegant form
\[\sum_{\substack{{k,l}\\k<l}}(\tilde x_k-x_k)(\tilde x_l-x_l)\leq
0.\] However,
\[0=\biggl(\sum_k (\tilde x_k-x_k)\biggr)^2=\sum_{k,l}(\tilde x_k-x_k)(\tilde x_l-x_l)
=\sum_k(\tilde x_k-x_k)^2 +2\sum_{\substack{{k,l}\\k<l}}(\tilde
x_k-x_k)(\tilde x_l-x_l),\] so that
\[\sum_{\substack{{k,l}\\k<l}}(\tilde x_k-x_k)(\tilde
x_l-x_l)=-\frac{1}{2}\sum_k(\tilde x_k-x_k)^2\leq 0\] as required.

We now verify, under the assumption $\min_i c_i>0$, that equation in
the theorem holds if and only if $x_{ik}=x_{jk}$ for all $i,j,k$ or
there exists some $l$ such that $x_{ik}=0$ for all $i$ and all
$k\neq l$. The ``if" part is easy, so we focus on the ``only if"
part. Inspecting the above argument carefully, we can see that
equation can hold only if for any $1\leq i<j\leq n$ the numbers
$x_k:=x_{ik}$, $x'_k:=x_{jk}$ satisfy
\[\lambda\sum_{\substack{{k,l,m}\\k<l}}x_mx'_kx'_l=
\sum_{\substack{{k,l,m}\\k<l}}x_mx'_kx'_l=\sum_{\substack{{k,m}\\m\neq
k}}x_k^2x'_m+\sum_{\substack{{k,l,m}\\k<l\\m\neq k,l}}x_kx_lx'_m,\]
where $\lambda$ is as before. If $x'_kx'_l=0$ for all $k<l$, then
$x_k^2x'_m=0$ for all $k\neq m$, i.e. $x_kx'_l=0$ for all $k\neq l$.
Otherwise $\lambda=1$ and $x_k=\tilde x_k=x'_k$ for all $k$ by the
above argument. In other words, equation in the theorem can hold
only if for any $i\neq j$ we have $x_{ik}x_{jl}=0$ for all $k\neq l$
or we have $x_{ik}=x_{jk}$ for all $k$. If there exist $j,l$ such
that $x_{jk}=0$ for all $k\neq l$, then $x_{jl}>0$ and for any
$i\neq j$ both alternatives imply $x_{ik}=0$ for all $k\neq l$,
hence we are done. Otherwise the first alternative cannot hold for
any $i\neq j$, so we are again done.
\end{proof}

\end{document}